\documentclass{article}
\usepackage[margin=3cm]{geometry}
\usepackage{amssymb,amsmath,amsthm,thmtools}
\usepackage{hyperref}
\usepackage[boxruled]{algorithm2e}
\usepackage[shortlabels]{enumitem}
\usepackage[dvipsnames]{xcolor}
\usepackage[round]{natbib}

\allowdisplaybreaks

\newcommand{\be}{\mathbf{e}}
\newcommand{\bff}{\mathbf{f}}

\newcommand{\bp}{\mathbf{p}}
\newcommand{\bq}{\mathbf{q}}

\newcommand{\bx}{\mathbf{x}}
\newcommand{\by}{\mathbf{y}}

\newcommand{\algout}[1]{{\cal A}(X,{#1})}

\newcommand{\eqnref}[1]{Eq.~\eqref{#1}}

\newcommand{\xx}{\mathbf{x}}
\newcommand{\xs}[1]{\xx^{({#1})}}
\newcommand{\yy}{\mathbf{y}}
\newcommand{\zz}{\mathbf{z}}
\newcommand{\M}{\mathbf{M}}
\newcommand{\Rn}{{\R^n}}
\newcommand{\Rm}{{\R^m}}
\newcommand{\Rk}{{\R^k}}
\newcommand{\R}{\mathbb{R}}
\newcommand{\Rext}{\overline{\R}}
\newcommand{\Hfcn}[1]{H_{#1}}
\newcommand{\Hk}{\Hfcn{k}}
\newcommand{\ffcn}[1]{f_{#1}}
\newcommand{\fI}{\ffcn{I}}
\newcommand{\Gfcn}[1]{G_{#1}}
\newcommand{\GI}{\Gfcn{I}}
\newcommand{\indX}{\delta_{X}}
\newcommand{\xopt}{\xx^*}
\newcommand{\xeps}{\xx_{\eps}}

\newcommand{\xhat}{\hat{\xx}}
\newcommand{\hatx}{\hat{x}}
\newcommand{\xopti}{x^*}

\newcommand{\braces}[1]{\left\{#1\right\}}
\newcommand{\parens}[1]{\left(#1\right)}

\newcommand{\transvec}[1]{\left(#1\right)^\top}

\newcommand{\bbR}{\mathbb{R}}

\newcommand{\hbx}{\hat{\bx}}

\DeclareMathOperator{\lexmax}{lexmax}

\newcommand{\eps}{\varepsilon}
\newcommand{\goes}{\rightarrow}

\newtheorem{theorem}{Theorem}
\newtheorem{lemma}{Lemma}

\newtheorem{definition}{Definition}

\newtheorem{proposition}{Proposition}

\providecommand{\norm}[1]{\left\lVert#1\right\rVert}
\newcommand{\barpair}[2]{(#1 ~|~ #2)}

\newtheoremstyle{mysubproof}
  {8pt plus 2pt minus 4pt} 
  {8pt plus 2pt minus 4pt} 
  {\normalfont}  
  {0pt}       
  {\itshape}  
  {.}         
  {.5em}      
  {}          

\newtheoremstyle{mysubclaim}
  {8pt plus 2pt minus 4pt} 
  {8pt plus 2pt minus 4pt} 
  {\normalfont}  
  {0pt}       
  {\itshape}  
  {.}         
  {.5em}      
  {}          

\declaretheorem[title=Claim,within=theorem,style=mysubclaim]{claimpx}

\declaretheorem[title={\hskip0pt Proof of claim},numbered=no,style=mysubproof,qed=$\lozenge$]{proofx}

\begin{document}

\title{Lexicographic Optimization: Algorithms and Stability}
\author{Jacob Abernethy$^{1, 2}$ \and Robert E. Schapire$^3$ \and Umar Syed$^2$}
\date{%
    $^1$College of Computing, Georgia Institute of Technology\\%
    $^2$Google Research\\%
    $^3$Microsoft Research\\[2ex]%
}

\maketitle

\begin{abstract}%
A lexicographic maximum of a set $X \subseteq  \bbR^n$ is a vector in $X$ whose smallest component is as large as possible, and subject to that requirement, whose second smallest component is as large as possible, and so on for the third smallest component, etc. Lexicographic maximization has numerous practical and theoretical applications, including fair resource allocation, analyzing the implicit regularization of learning algorithms, and characterizing refinements of game-theoretic equilibria. We prove that a minimizer in $X$ of the exponential loss function $L_c(\bx) = \sum_i \exp(-c x_i)$ converges to a lexicographic maximum of $X$ as $c \goes \infty$, 
provided that $X$ is
\emph{stable} in the sense that a well-known iterative method for
finding a lexicographic maximum of $X$ cannot be made to fail simply
by reducing the required quality of each iterate by an arbitrarily
tiny degree.
Our result holds for both near and exact minimizers of the exponential loss, while earlier convergence results made much stronger assumptions about the set $X$ and only held for the exact minimizer. We are aware of no previous results showing a connection between the iterative method for computing a lexicographic maximum and exponential loss minimization.
We show that every convex polytope is stable, but that there exist
compact, convex sets that are not stable.
We also provide the first analysis of the convergence rate of an exponential loss minimizer (near or exact) and discover a curious dichotomy: While the two smallest components of the vector converge to the lexicographically maximum values very quickly (at roughly the rate $\frac{\log n}{c}$), all other components can converge arbitrarily slowly.
\end{abstract}

\section{Introduction}
\label{sec:intro}

A \emph{lexicographic maximum} of a set $X \subseteq \bbR^n$ is a vector in $X$ that is at least as large as any other vector in $X$ when sorting their components in non-decreasing order and comparing them lexicographically. For example, if
\[
X = \left\{\bx_1 = \begin{pmatrix} 5\\ 2\\ 4 \end{pmatrix}, \bx_2 = \begin{pmatrix} 2\\ 6\\ 3 \end{pmatrix}, \bx_3 = \begin{pmatrix} 8\\ 7\\ 1 \end{pmatrix}\right\}
\]
then the lexicographic maximum of $X$ is $\bx_1$, since both $\bx_1$ and $\bx_2$ have the largest smallest components, and $\bx_1$ has a larger second smallest component than $\bx_2$. Infinite and unbounded sets can also contain a lexicographic maximum. 

One of the first applications of lexicographic maximization was fair bandwidth allocation in computer networks \citep{hayden1981voice, MoWalrand98, LeBoudec2000}. Lexicographic maximimization avoids a `rich-get-richer' allocation of a scarce resource, since a lexicographic maximum is both Pareto optimal and has the property that no component can be increased without decreasing another \emph{smaller} component. More recently, \cite{diana2021lexicographically} considered lexicographic maximization as an approach to fair regression, where the components of the lexicographic maximum represent the performance of a model on different demographic groups. \cite{bei2022truthful} studied the ``cake sharing'' problem in mechanism design and showed that assigning a lexicographically maximum allocation to the agents is a truthful mechanism.

\cite{Rosset2004} and \cite{Nacson2019} used lexicographic maximization to analyze the implicit regularization of learning algorithms that are based on minimizing an objective function. In particular, they showed for certain objective functions and model classes that the  vector of model predictions converges to a lexicographic maximum.

In game theory, \cite{dresher1961mathematics} described an equilibrium concept in which each player's payoff vector is a lexicographic maximum of the set of their possible payoff vectors. \cite{van1991stability} showed that in a zero-sum game this concept is equivalent to a \emph{proper equilibrium}, a well-known refinement of a Nash equilibrium. Lexicographic maximization is also used to define the \emph{nucleolus} of a cooperative game \citep{schmeidler1969nucleolus}.

Two methods for computing a lexicographic maximum have been described in the literature. The first method is an iterative algorithm that solves $n$ optimization problems, with the $i$th iteration computing the $i$th smallest component of the lexicographic maximum, and is guaranteed to find the lexicographic maximum if one exists. This method is often called the \emph{progressive filling algorithm} \citep{bertsekas2021data}. The second method minimizes the \emph{exponential loss}, which is defined
\begin{align}\label{eq:exploss}
    L_c(\bx) = \sum_{i=1}^n \exp(-c x_i),
\end{align}
where $\bx \in \bbR^n$ and $c \ge 0$. Previous work has shown that a minimizer in $X$ of $L_c(\bx)$ may converge to a lexicographic maximum of $X$ as $c \goes \infty$, but the conditions placed on $X$ to ensure convergence were quite restrictive. \citet{MoWalrand98}, \cite{LeBoudec2000} and \cite{Rosset2004} all assumed that $X$ is a convex polytope, while \citet{Nacson2019} assumed that $X$ is the image of a simplex under a continuous, positive-homogeneous mapping. These results were all based on asymptotic analyses, and did not establish any bounds on the rate of convergence to a lexicographic maximum. As far as we know, no previous work has drawn a connection between the two methods for computing a lexicographic maximum.

\paragraph{Our contributions:} The iterative algorithm described above is guaranteed to find a lexicographic maximum of a set if the optimization problems in each iteration are solved exactly. But if the optimization problems are solved only approximately, then the output of the algorithm can be far from a lexicographic maximum, even if the approximation error is arbitrarily small (but still non-zero). We define a property called \emph{lexicographic stability}, which holds for a set $X \subseteq \bbR^n$ whenever this pathological situation does not occur, and prove that it has an additional powerful implication: Any vector $\bx_c \in X$ that is less than $\exp(-O(c))$ from the minimum value of $L_c(\bx)$ converges to a lexicographic maximum of $X$ as $c \goes \infty$. By proving convergence for all non-pathological sets, we significantly generalize existing convergence criteria for exponential loss minimization. 

We show that all convex polytopes are stable, thereby subsuming most of
the previous works mentioned above.
On the other hand, we show that sets that are
convex and compact but not polytopes need not be stable, in general.
We also show that our convergence result does not hold generally when
stability is not assumed by constructing a set $X$ that is not lexicographically stable, and for which $\bx_c$ is bounded away from a lexicographic maximum of $X$ for all sufficiently large $c$, even if $\bx_c$ is an exact minimizer of the exponential loss. 

We also study the rate at which $\bx_c$ approaches a lexicographic maximum, and find a stark discrepancy for different components of $\bx_c$. The smallest and second smallest components of $\bx_c$ are never more than $O\left(\frac{\log n}{c}\right)$ below their lexicographically maximum values, even if $X$ is not lexicographically stable. However, all other components of $\bx_c$ can remain far below their lexicographically maximum values for arbitrarily large $c$, even if $\bx_c$ is an exact minimizer of the exponential loss and $X$ is lexicographically stable with a seemingly benign structure (it can be a single line segment).

Finally, we prove that the multiplicative weights algorithm \citep{littlestone1994weighted,FreundSc99} is guaranteed to converge to the lexicographic maximum of a convex polytope, essentially because it minimizes the exponential loss. While \cite{Syed2010} proved that the multiplicative weights algorithm diverges from the lexicographic maximum when the learning rate is constant, we guarantee convergence by setting the learning rate to $O(1/t)$ in each iteration $t$.

\paragraph{Additional related work:}
\label{sec:related}

Lexicographic maximization is often applied to multi-objective optimization problems where, for a given function $\bff : \Theta \rightarrow \bbR^n$, the goal is to find $\theta^* \in \Theta$ such that $\bff(\theta^*)$ is a lexicographic maximum of the set $X = \{\bff(\theta) : \theta \in \Theta\}$ \citep{luss1986resource}. As in our work, this approach sorts the components of each vector before performing a lexicographic comparison, which contrasts with other work in which an ordering of the dimensions is fixed in advance \citep{sherali1983preemptive, martinez1987lexicographical}.

\cite{diana2021lexicographically} and \cite{henzinger2022leximax} critiqued the lexicographic maximum as a fairness solution concept because of its potential for instability when subject to small perturbations. They gave examples demonstrating that instability can occur, but did not relate the instability to the problem of computing a lexicographic maximum via loss minimization. The primary objective of our work is to provide a much fuller characterization of this instability and to explore its implications.

In the analytical framework used by \cite{Rosset2004} and \cite{Nacson2019}, larger values of $c$ for the exponential loss $L_c(\bx)$ correspond to weaker explicit regularization by a learning algorithm, so understanding the behavior of the minimizer of $L_c(\bx)$ as $c \goes \infty$ helps to characterize the algorithm's implicit regularization.

\citet{hartman2023leximin} appear to contradict one of our key negative results by showing that if the iterative algorithm uses an approximate solver in each iteration then its output will always be close to a lexicographic maximum, provided that the approximation error is sufficiently small (see their Theorem 9). However, this apparent discrepancy with our results is in fact due only to us using a different and incompatible definition of ``closeness.'' See Appendix \ref{sec:hartman} for a discussion.

\section{Preliminaries}
\label{sec:prelim}

For any non-negative integer $n$ let $[n] := \{1, \ldots, n\}$, and note that $[0]$ is the empty set. Let $\bbR_{\ge 0} := \{x \in \bbR : x \ge 0\}$ be the non-negative reals. Let $x_i$ be the $i$th component of the vector $\bx \in \bbR^n$. Let $\norm{X}_\infty := \sup_{\bx \in X} \max_i |x_i|$ be the largest $\ell_\infty$ norm of any vector in $X \subseteq \bbR^n$.

For each $i \in [n]$ let $\sigma_i : \bbR^n \rightarrow \bbR$ be the \emph{$i$th sorting function}, such that $\sigma_i(\bx)$ is the $i$th smallest component of $\bx \in \bbR^n$. For example, if $\bx = (2, 1, 2)^\top$ then $\sigma_1(\bx) = 1$, $\sigma_2(\bx) = 2$ and $\sigma_3(\bx) = 2$. We define a total order $\geq_\sigma$ on vectors in $\bbR^n$ as follows: for any points $\bx, \bx' \in \bbR^n$ we say that $\bx \geq_\sigma \bx'$ if and only if $\bx = \bx'$ or $\sigma_i(\bx) > \sigma_i(\bx')$ for the smallest $i \in [n]$ such that $\sigma_i(\bx) \neq \sigma_i(\bx')$.

\begin{definition} \label{defn:lexopt} A \emph{lexicographic maximum} of a set $X \subseteq \bbR^n$ is a vector $\bx^* \in X$ for which $\bx^* \geq_\sigma \bx$ for every $\bx \in X$.
Let $\lexmax X$ be the set of all lexicographic maxima of $X$.
\end{definition}
While $\lexmax X$ can be empty, this can only occur if $X$ is empty or not compact (see Theorem~\ref{thm:exists}).

For notation, we always write $\bx^*$ to denote a lexicographic maximum of $X$. Also, for all $c, \gamma \ge 0$ we write $\bx_{c, \gamma}$ to denote an arbitrary vector in $X$ that satisfies
\[
L_c(\bx_{c, \gamma}) \le \inf_{x \in X} L_c(\bx) + \gamma\exp(-c\norm{X}_\infty).
\]
In other words, $\bx_{c, \gamma}$ is a near minimizer in $X$ of the exponential loss if $\gamma$ is small, and $\bx_{c, 0}$ is an exact minimizer. The notation $\bx^*$ and $\bx_{c, \gamma}$ suppresses the dependence on $X$, which will be clear from context. While $\bx^*$ and $\bx_{c, 0}$ do not exist in every set $X$, we are only interested in cases where they do, and implicitly make this assumption throughout our analysis. An exception is when we construct a set $X$ to be a counterexample. In these cases we explicitly prove that $\bx^*$ and $\bx_{c, 0}$ exist. Also, a set may contain multiple vectors that satisfy the definitions of $\bx^*$ or $\bx_{c, \gamma}$, and our results apply no matter how they are chosen.

Our goal is to characterize when $\bx_{c, \gamma}$ is ``close'' to $\bx^*$, where we use the following definition of closeness.

\begin{definition} \label{defn:distort} The \emph{lexicographic distortion} between $\bx, \bx' \in \bbR^n$ with respect to $I \subseteq [n]$ is 
\[
d_I\barpair{\bx}{\bx'} \triangleq \max_{k \in I} [ \max\{0,\sigma_k(\bx) - \sigma_k(\bx')\}].
\]
If $I = [n]$ we abbreviate $d\barpair{\bx}{\bx'} \triangleq d_I\barpair{\bx}{\bx'}$.
\end{definition}

Lexicographic distortion is useful for quantifying the closeness of $\bx_{c,\gamma}$ to $\bx^*$ because $d\barpair{\bx^*}{\bx_{c,\gamma}} = 0$ if and only if $\bx_{c,\gamma} \in \lexmax X$. It is important to note, however, that $d_I\barpair{\cdot}{\cdot}$ is \textit{not} a symmetric function, and most typically $d_I\barpair{\bx}{\bx'} \ne d_I\barpair{\bx'}{\bx}$

\paragraph{Problem statement:} We want to describe conditions on $X \subseteq \bbR^n, I \subseteq [n], c \ge 0$ and $\gamma \ge 0$ which ensure that $d_I\barpair{\bx^*}{\bx_{c,\gamma}}$ is almost equal to zero. We are particularly interested in cases where $I 
= [n]$, since this implies that $\bx_{c,\gamma}$ is close to a lexicographic maximum. Also, we want to identify counterexamples where $d_I\barpair{\bx^*}{\bx_{c,\gamma}}$ is far from zero. Since $d\barpair{\bx^*}{\bx_{c,\gamma}} \ge d_I\barpair{\bx^*}{\bx_{c,\gamma}}$ for all $I \subseteq [n]$, this implies that $\bx_{c,\gamma}$ is far from a lexicographic maximum.

\section{Basic properties}

We prove several basic properties of lexicographic maxima that will be useful in our subsequent analysis. We first show that conditions which suffice to ensure that a subset of $\bbR$ contains a maximum also ensure that a subset of $\bbR^n$ contains a lexicographic maximum.

\begin{theorem} \label{thm:exists} If $X  \subseteq \bbR^n$ is non-empty and compact then $\lexmax X$ is non-empty.\end{theorem}
\begin{proof}
Define $X_0, \ldots, X_n$ and $\Sigma_1, \ldots, \Sigma_n$ recursively as follows: Let $X_0 = X$, $\Sigma_i = \{\sigma_i(\bx) : \bx \in X_{i-1}\}$ and $X_i = \{\bx \in X_{i-1} : \sigma_i(\bx) \ge \sup \Sigma_i\}$. We will prove by induction that each $X_i$ is non-empty and compact, which holds for $X_0$ by assumption. Since $X_n = \lexmax X$ this will complete the proof.

By Theorem \ref{thm:continuitysorting} in
Appendix~\ref{sec:sorting-cont},
each sorting function $\sigma_i$ is continuous. If $X_{i-1}$ is non-empty and compact then $\Sigma_i$ is non-empty and compact, since it is the image of a compact set under a continuous function. Therefore $\sup \Sigma_i \in \Sigma_i$. If $X_{i-1}$ is non-empty and compact and $\sup \Sigma_i \in \Sigma_i$ then $X_i$ is non-empty and closed, since it is the pre-image of a compact set under a continuous function. Also, $X_i$ is bounded, since $X_i \subseteq X_{i-1}$, and therefore $X_i$ is compact.\end{proof}

Furthermore, if $X$ is also convex, we are assured that the lexicographic maximum of $X$ is unique.

\begin{theorem}  \label{thm:lexopt-unique}
If $X \subseteq \bbR^n$ is non-empty, compact and convex then $| \lexmax X | = 1$.\end{theorem}

\begin{proof}
Suppose $X$ is nonempty, compact and convex.
By Theorem~\ref{thm:exists}, $|\lexmax X| \geq 1$, so it
only remains to show that $|\lexmax X| \leq 1$.
Suppose, towards a contradiction, there exist distinct points
$\bx,\by \in X$ that are both lexicographic maxima.
We assume without loss of generality that the coordinates of $\bx$ are
sorted in nondecreasing order and, further, that on any segment of
``ties'' on the sorted $\bx$, the corresponding segment in $\by$ is
nondecreasing.
Thus, for $i,j\in [n]$, if $i\leq j$ then
$x_i\leq x_j$, and in addition,
if $x_i=x_j$ then $y_i \leq y_j$.
Since both $\bx$ and $\by$ are lexicographic maxima, it follows that
$\sigma_i(\by)=\sigma_i(\bx)=x_i$ for $i\in [n]$.

Let $k$ be the smallest index on which $\bx$ and $\by$ differ
(so that $x_i=y_i$ for $i<k$ and $x_k\neq y_k$).
Let $I=[k-1]$.
Then
$\sigma_i(\by)=\sigma_i(\bx)=x_i=y_i$ for
$i\in I$.
Therefore, $\sigma_k(\by)$ is the smallest of the remaining
components of $\by$, implying
\begin{equation}  \label{eq:thm:lexopt-unique:1}
  y_k
  \geq
  \min\{y_k,\ldots,y_n\}
  =
  \sigma_k(\by)
  =
  \sigma_k(\bx)
  =
  x_k,
\end{equation}
and so that $y_k>x_k$. Let $\zz=(\bx+\by)/2$,
which is in $X$ since $X$ is convex.
We consider the components of $\zz$ relative to $x_k$.
Let $i\in [n]$. If $i<k$ then $z_i = x_i \leq x_k$
(since $x_i=y_i$). If $i=k$ then $z_k = (y_k+x_k)/2 > x_k$
since $y_k>x_k$.

Finally, suppose $i>k$, implying $x_i\geq x_k$.
If $x_i>x_k$ then
$z_i=(y_i+x_i)/2 > x_k$ since
$y_i\geq x_k$
(by \eqnref{eq:thm:lexopt-unique:1}).
Otherwise, $x_i=x_k$, implying, by how the components are
sorted, that
$y_i\geq y_k>x_k$;
thus, again,
$z_i=(y_i+x_i)/2>x_k$.

To summarize, $z_i=y_i=x_i\leq x_k$ if $i\in I$,
and $z_i>x_k$ if $i\not\in I$.
It follows that
$\sigma_i(\zz)=\sigma_i(\bx)=x_i$ for $i=1,\ldots,k-1$,
and that
$\sigma_k(\zz)=\min\{z_k,\ldots,z_n\}>x_k=\sigma_k(\bx)$.
However, this contradicts that $\bx$ is a lexicographic maximum.
\end{proof}

\section{Computing a lexicographic maximum}

Algorithm \ref{alg:iterative} below is an iterative procedure for computing a lexicographic maximum of a set $X \subseteq \bbR^n$. In each iteration $k$, Algorithm~\ref{alg:iterative} finds a vector in $X$ with the (approximately) largest $k$th smallest component, subject to the constraint that its $k-1$ smallest components are at least as large as the $k-1$ smallest components of the vector from the previous iteration. The quality of the approximation is governed by a tolerance parameter $\eps$. Since the optimization problem in each iteration can have multiple solutions, the output of Algorithm~\ref{alg:iterative} should be thought of as being selected arbitrarily from a set of possible outputs.
We write $\algout{\eps}$ for the set of all possible outputs when
Algorithm~\ref{alg:iterative} is run on input $(X,\eps)$.

Algorithm~\ref{alg:iterative} and close variants have been described many times in the literature. In general, the optimization problem in each iteration can be difficult to solve, especially due to the presence of sorting functions in the constraints. Consequently, much previous work has focused on tractable special cases where the optimization problem can be reformulated as an equivalent linear or convex program \citep{luss1999equitable, miltersen2006computing}. Other authors have used Algorithm \ref{alg:iterative} with tolerance parameter $\eps = 0$ as the definition of a lexicographic maximum \citep{van1991stability,Nacson2019, diana2021lexicographically}. Theorem~\ref{thm:equiv} explains the relationship between Definition~\ref{defn:lexopt} and Algorithm~\ref{alg:iterative} when $\eps = 0$.

\RestyleAlgo{ruled}
\SetKw{KwRet}{Return:}
\begin{algorithm}\caption{Compute a lexicographic maximum. \label{alg:iterative}}
\KwIn{Set $X \subseteq \bbR^n$, tolerance $\eps \ge 0$.}
\For{$k = 1, \ldots, n$}{ Find an $\eps$-optimal solution $\bx^{(k)}$ to the optimization problem: \begin{align*}
    \max_{\bx \in X} &~\sigma_k(\bx)\\
    \textrm{subject to} &~ \sigma_i(\bx) \ge \sigma_i(\bx^{(k-1)})\textrm{ for all }i \in [k-1]
\end{align*}
}
\KwRet{$\bx^{(n)}$.}
\end{algorithm}

\begin{theorem} \label{thm:equiv}
Let $X \subseteq \bbR^n$. A vector $\hbx \in \bbR^n$ is a possible output of Algorithm \ref{alg:iterative} on input $(X, 0)$ if and only if $\hbx \in \lexmax X$.
That is, $\algout{0}=\lexmax X$.
\end{theorem}

\begin{proof}
Let $\bx^* \in \lexmax X$. We prove by induction that in each iteration $k$ of Algorithm \ref{alg:iterative} we have $\sigma_i(\bx^{(k)}) = \sigma_i(\bx^*)$ for all $i \in [k]$. Setting $k = n$ proves the theorem. For the base case $k = 1$, observe that the algorithm assigns $\bx^{(1)} \in \arg \max_{\bx \in X} \sigma_1(\bx)$. Therefore $\sigma_1(\bx^{(1)}) = \sigma_1(\bx^*)$, by the definition of $\bx^*$. In each iteration $k > 1$ the algorithm assigns 
\begin{align*}
    \bx^{(k)} \in  & \arg \max_{\bx \in X} \sigma_k(\bx) \\
    & \textrm{ subject to } \sigma_i(\bx) \ge \sigma_i(\bx^*) \textrm{ for all } i \in [k - 1], 
\end{align*}
where the constraints are implied by the inductive hypothesis. Therefore $\sigma_i(\bx^{(k)}) = \sigma_i(\bx^*)$ for all $i \in [k]$, again by the definition of $\bx^*$.\end{proof}

While Algorithm \ref{alg:iterative} can only find a lexicographic maximum if one exists in $X$, we recall from Theorem \ref{thm:exists} that this holds whenever $X$ is non-empty and compact.

\cite{diana2021lexicographically} and \cite{henzinger2022leximax} proposed $\algout{\eps}$ (or minor variants thereof) as the definition of the $\eps$-approximate lexicographic maxima of $X$. However, they observed that $\algout{\eps}$ may nonetheless contain vectors that are far from any lexicographic maximum, even if $\eps$ is very small (but still non-zero). In the next section we formally characterize this phenomenon, and in the rest of the paper we explore its implications.

\section{Lexicographic stability}

Theorem~\ref{thm:equiv} states that Algorithm~\ref{alg:iterative} outputs a lexicographic maximum (assuming one exists) if the optimization problem in each iteration of the algorithm is solved exactly.  In practice, the optimization problems will be solved by a numerical method up to some tolerance $\eps > 0$, with smaller values of $\eps$ typically requiring longer running times. Ideally, we would like the quality of the output of Algorithm~\ref{alg:iterative} to vary smoothly with $\eps$, and if this happens for a set $X$ then we say that $X$ is \emph{lexicographically stable}.

\begin{definition}
\label{defn:stable} A set $X \subseteq \bbR^n$ is \emph{lexicographically stable} if for all $\delta > 0$ there exists $\eps > 0$ such that for all $\hbx \in \algout{\eps}$ and $\bx^* \in \lexmax X$ we have $d\barpair{\bx^*}{\hbx} < \delta$.
\end{definition}

Definition~\ref{defn:stable} says that if a set $X$ is lexicographically stable, then in order to find an arbitrarily good approximation of its lexicographic maximum, it suffices to run Algorithm~\ref{alg:iterative} with a sufficiently small tolerance parameter $\eps > 0$. On the other hand, if $X$ is not lexicographically stable, then no matter how small we make $\eps > 0$ the output of Algorithm~\ref{alg:iterative} can be far from a lexicographic maximum.

A sufficient (but not necessary) condition for a set $X \subseteq \bbR^n$ to be lexicographically stable is that $X$ is a convex polytope ({i.e.}, the convex hull of finitely many points in $\bbR^n$). 

\begin{theorem} \label{thm:convex} If $X \subseteq \bbR^n$ is a convex polytope then $X$ is lexicographically stable. \end{theorem}
\begin{proof}[Proof sketch]
We outline the main steps of the proof.
Detailed justification for these steps is given in
Appendix~\ref{sec:convex-polytop-stable}.
We suppose $X$ is a convex polytope and that $\xopt$ is its unique
lexicographic maximum.

On each round~$k$, Algorithm~\ref{alg:iterative} approximately solves
an optimization problem whose value is given by some function
$\Hk(\xs{k-1})$ so that $\xs{k}$ must satisfy
$\sigma_k(\xs{k})\geq \Hk(\xs{k-1}) - \eps$.
In fact, we show that a point $\xx$ is a possible output on $(X,\eps)$
if and only if it satisfies
$\sigma_k(\xx)\geq \Hk(\xx) - \eps$
for all $k\in [n]$.
This function $\Hk$ can be lower-bounded by another function $\GI$,
which, if $X$ is a convex polytope, is concave and lower
semicontinuous.

If $X$ is not lexicographically stable, then there exists sequences
$\xx_t$ and $\eps_t>0$ with $\eps_t\rightarrow 0$ and
$\xx_t\in\algout{\eps_t}$ such that $\xx_t\rightarrow\xhat$ for some
$\xhat\in X$ but $\xhat\neq\xopt$.
We can re-index points in $X$ so that the components of $\xhat$ are
sorted, and furthermore, that when there are ties, they are
additionally sorted according to the components of $\xopt$.
Let $k$ be the first component for which $\hatx_k\neq\xopti_k$.
We show this implies $\GI(\xhat) > \hatx_k$.
Combining facts, we then have
\begin{align*}
  \hatx_k
  <\; &
  \GI(\xhat)
  \leq
  \liminf_{t\rightarrow\infty} \GI(\xx_t)
  \\
  &\leq
  \liminf_{t\rightarrow\infty} \Hk(\xx_t)
  \leq
  \liminf_{t\rightarrow\infty} [\sigma_k(\xx_t) + \eps_t]
  =
  \hatx_k,
\end{align*}
a contradiction.
\end{proof}

Combining Theorem~\ref{thm:convex} with Theorem~\ref{thm:main} below recovers results due to \cite{MoWalrand98}, \cite{LeBoudec2000} and \cite{Rosset2004} which show that an exponential loss minimizer in $X$ converges to a lexicographic maximum of $X$ when $X$ is a convex polytope.

Fairly simple non-convex sets that contain a lexicographic maximum but are not lexicographically stable exist. We provide an example in Theorem~\ref{thm:sharp_lower},
one that also shows how, in such cases, minimizing
exponential loss might not lead to a lexicographic maximum.

Theorem~\ref{thm:convex} shows that a nonempty set $X$ is
lexicographically stable if it is a convex polytope, a condition that
implies that $X$ is also convex and compact.
The theorem still holds, by the same proof, with the weaker requirement
that $X$ is convex, compact and \emph{locally simplicial}, a property
defined in \citet[page~84]{ROC}.
However, the theorem is false, in general, if we only require that $X$
be convex and compact.

As an example, in $\R^3$, let $Z=[-1,0]\times[0,1]\times[0,1]$, and
let
\begin{equation}   \label{eqn:compact-X-not-lex-stab}
  X
  =
  \braces{ \xx\in Z :
            x_1 (x_3 - 1) \geq x_2^2
         }.
\end{equation}
This set is compact and convex, and
has a unique lexicographic maximum, namely,
$\xopt=\transvec{0,0,1}$.
For $\eps\in (0,1)$, let $\xeps=\transvec{-\eps^2,\eps/2,3/4}$.
It can be shown that $\xeps\in\algout{\eps}$.
It follows that $X$ is not lexicographically stable since
$\sigma_3(\xeps)=3/4$ for all $\eps>0$ while
$\sigma_3(\xopt)=1$.
(Details are given in
Appendix~\ref{sec:convex-not-stable}.)

\section{Convergence analysis of exponential loss minimization}
\label{sec:convergence}

In this section we study conditions under which a near or exact minimizer $\bx_{c, \gamma} \in X$ of the exponential loss $L_c(\bx)$ converges to a lexicographic maximum $\bx^* \in \lexmax X$ as $c \goes \infty$. To see why convergence should be expected, note that when $c$ is large, the dominant term of $L_c(\bx)$ corresponds to the smallest component of $\bx$, since the function $x \mapsto \exp(-cx)$ decreases very quickly. Therefore minimizing $L_c(\bx)$ will tend to make this term as large as possible. Further, among vectors $\bx$ that maximize their smallest component, the dominant term in $L_c(\bx)$ corresponds to the second smallest component of $\bx$, if we ignore terms that are equal for all vectors. In general, when $c$ is large, the magnitudes of the terms in $L_c(\bx)$ decrease sharply when they are sorted in increasing order of the components of $\bx$, and this situation tends to favor a minimizer of $L_c(\bx)$ that is also a lexicographic maximum.
Although such reasoning is intuitive, proving convergence to a
lexicographic maximum can be challenging;
indeed, convergence need not hold for every set $X$,
as will be seen shortly.

\subsection{Asymptotic results}

Theorem~\ref{thm:main} states our main convergence result: If $X$ is lexicographically stable then $\bx_{c, \gamma}$ converges to a lexicographic maximum $\bx^*$ as $c \goes \infty$, provided that $\gamma \in [0, 1)$. In other words, in the contrapositive, if a near minimizer in $X$ of $L_c(\bx)$ fails to converge to a lexicographic maximum as $c \goes \infty$, then Algorithm~\ref{alg:iterative} can fail to find a good approximation of a lexicographic maximum for any tolerance $\eps > 0$.

\begin{theorem} \label{thm:main} Let $X \subseteq \bbR^n$ and $\gamma \in [0, 1)$. If $X$ is lexicographically stable then for all $\bx^* \in \lexmax X$
\[
\lim_{c \goes \infty} d\barpair{\bx^*}{\bx_{c, \gamma}} = 0.
\]
\end{theorem}

Theorem~\ref{thm:main}'s requirement that $\gamma \in [0, 1)$ ensures that $\bx_{c, \gamma}$ is less than $\exp(-c\norm{X}_\infty)$ from the minimum value of $L_c(\bx)$. To see why this condition aids convergence to a lexicographic maximum, consider that the smallest possible value of any term in $L_c(\bx)$ is $\exp(-c\norm{X}_\infty)$. So if the minimization error were larger than this value then $\bx_{c, \gamma}$ may not make every term of $L_c(\bx_{c, \gamma})$ small, which in turn could cause $\bx_{c, \gamma}$ to be far from a lexicographic maximum.

Before proving Theorem~\ref{thm:main}, we introduce some additional notation and a key lemma. For any $X \subseteq \bbR^n$, $\bx \in \bbR^n$ and $k \in \{0\} \cup [n]$ let
\begin{equation}
X_k(\bx) = \{\bx' \in X : \sigma_i(\bx') \ge \sigma_i(\bx) \textrm{ for all } i \in [k]\} \label{eq:constraints}
\end{equation}
be the set of all vectors in $X$ whose $k$ smallest components are at least as large as the $k$ smallest components of $\bx$. Note that $X_0(\bx) = X$. Also, if $\bx^{(k-1)}$ is the vector selected in iteration $k-1$ of Algorithm~\ref{alg:iterative}, then $X_{k-1}(\bx^{(k-1)})$ is the set of feasible solutions to the optimization problem in iteration $k$ of the algorithm. Lemma~\ref{thm:helper} below, which is key to our convergence results and proved in Appendix~\ref{sec:lem-helper}, states that if $\bx_{c, \gamma}$ is selected in any iteration of Algorithm~\ref{alg:iterative} then it is a good solution for the next iteration when $c$ is large.

\begin{lemma} \label{thm:helper} For all $X \subseteq \bbR^n, \gamma \in [0, 1), c > 0$ and $k \in [n]$
\[
\sigma_k(\bx_{c,\gamma}) \ge \sup_{\bx \in X_{k-1}(\bx_{c, \gamma})} \sigma_k(\bx) - \frac1c \log\left(\frac{n - k + 1}{1 - \gamma}\right).
\]
\end{lemma}
We are now ready to prove Theorem~\ref{thm:main}.

\begin{proof}[Proof of Theorem~\ref{thm:main}] Say that a vector $\bx \in \bbR^n$ is \emph{$(X, \eps)$-valid} if it is a solution to all of the optimization problems in Algorithm~\ref{alg:iterative} when run on input $(X, \eps)$. In other words, if Algorithm~\ref{alg:iterative} is run on input $(X, \eps)$, then the algorithm can let $\bx^{(1)} = \bx, \bx^{(2)} = \bx, \ldots, \bx^{(n)} = \bx$. For all $c > 0$ let $\eps(c) = \frac{1}{c} \log \frac{n}{1 - \gamma}$. By Lemma~\ref{thm:helper} we have
\[
\sigma_k(\bx_{c, \gamma}) \ge \sup_{\bx \in X_{k-1}(\bx_{c, \gamma})} \sigma_k(\bx) - \eps(c)
\]
for all $k \in [n]$, which immediately implies that $\bx_{c, \gamma}$ is $(X, \eps(c))$-valid. Let $\{\delta_t\}$ be a positive sequence with $\lim_{t \goes \infty} \delta_t = 0$. Since $\eps(c)$ is a continuous function with range $(0, \infty)$, by Definition~\ref{defn:stable} for each $\delta_t$ there exists $c_t > 0$ such that for all $\hbx \in \algout{\eps(c_t)}$ and $k \in [n]$ we have
\[
\delta_t \ge \sigma_k(\bx^*) - \sigma_k(\hbx).
\]
Also, since $\algout{\eps} \subseteq \algout{\eps'}$ if $\eps < \eps'$, and $\eps(c)$ is a decreasing function, we can arrange $\{c_t\}$ to be an increasing sequence with $\lim_{t \goes \infty} c_t = \infty$.

Since $\bx_{c, \gamma}$ is $(X, \eps(c))$-valid we have for all $k \in [n]$
\[
\delta_t \ge \sigma_k(\bx^*) - \sigma_k(\bx_{c_t, \gamma}).
\]
By taking the limit superior of both sides we have for all $k \in [n]$
\[
0 \ge \limsup_{t \goes \infty} \left[\sigma_k(\bx^*) - \sigma_k(\bx_{c_t, \gamma})\right],
\]
and therefore 
\[
\lim_{t \goes \infty} \max_{k \in [n]}[\max\{0, \sigma_k(\bx^*) - \sigma_k(\bx_{c_t, \gamma})\}] = 0,
\] which proves the theorem.\end{proof}

The lexicographic stability requirement in Theorem~\ref{thm:main} cannot be relaxed without risking non-convergence. Theorem~\ref{thm:sharp_lower} below constructs a lexicographically unstable set for which the exact exponential loss minimizer is bounded away from the lexicographic maximum. The set is a piecewise linear path consisting of two adjoining line segments that is bounded, closed and connected, but not convex.

\begin{theorem} \label{thm:sharp_lower} For all $n \ge 8$ there exists a set $X \subseteq \bbR^n$ consisting of two line segments with a shared endpoint and satisfying $\norm{X}_\infty = 1$ such that for all $\bx^* \in \lexmax X$ and $c \ge 2$ 
\[
d\barpair{\bx^*}{\bx_{c,0}} \ge d_{\{n\}}\barpair{\bx^*}{\bx_{c,0}} \ge \frac12.
\]
\end{theorem}
\begin{proof}[Proof sketch]
The full construction and proof are given in
Appendix~\ref{sec:thm:sharp-lower}.
Briefly,
the lexicographic maximum $\bx^*$ in $X$ is a vector consisting of $0$ in the first $n - 1$ components and $1$ in the $n$th component.
For all $\eps > 0$, $X$ also includes a vector $\bx^{(\eps)}$ whose first component is $-\frac{\eps}{2}$, next $n-2$ components are $\frac{\eps}{4}$, and $n$th component is $\frac12$. We prove that for all $n \ge 8$ and $c \ge 2$ there exists $\eps > 0$ such that $L_c(\bx^{(\eps)}) < L_c(\bx^*)$, essentially because when $n$ is sufficiently large the lower loss on the middle $n-2$ components compensates for the higher loss on the first and last components. Observing that $\sigma_n(\bx^{(\eps)}) = \sigma_n(\bx^*) - \frac12$ completes the proof.
\end{proof}

From Theorem~\ref{thm:main} it immediately follows that the set constructed in Theorem~\ref{thm:sharp_lower} is not lexicographically stable. We can also give more direct intuition for why the set is unstable. The set $X$ in Theorem~\ref{thm:sharp_lower} consists of a ``good'' and a ``bad'' line segment, and the unique lexicographic maximum is a point on the ``good'' line segment. If the optimization problem in the first iteration of Algorithm~\ref{alg:iterative} is solved exactly, then the smallest component of the solution will be equal to $0$. In this case, every iteration of the algorithm will output a solution on the ``good'' line segment, since only points on that segment have a smallest component that is at least $0$. However, if the optimization problem in the first iteration is solved with a tolerance $\eps > 0$, then the smallest component of the solution can be as small as $-\eps$. In this case, every iteration of the algorithm will output a solution on the ``bad'' line segment, since the $2$nd, $3$rd, \ldots, $(n-1)$th smallest components of the points on that segment are larger than the corresponding components of the points on the ``good'' line segment. However, the largest component of each point on the ``good'' line segment is equal to $1$, while the largest component of each point on the ``bad'' line segment is equal to $\frac12$. As a result, the algorithm outputs a vector whose largest component has a value that is far from its lexicographic maximum.

Theorem \ref{thm:sharp_lower} provides a much stronger example of instability than the construction of \citet{diana2021lexicographically}, who showed that for all $\eps > 0$ there exists a set $X$ such that $\algout{\eps}$, the set of possible outputs on input $(X,\eps)$, contains an element that is far from the lexicographic maximum $\bx^*$ of $X$. By contrast, Theorem~\ref{thm:sharp_lower} reverses the order of the quantifiers, and shows that there exists a set $X$ such that for all $\eps > 0$ the set $\algout{\eps}$ contains an element that is far from $\bx^*$.

\subsection{Convergence rates}

Theorems~\ref{thm:upper} and~\ref{thm:linesegmentlowerbound} below give bounds on the rate at which a near or exact minimizer $\bx_{c, \gamma} \in X$ of the exponential loss function $L_c(\bx)$ converges to a lexicographic maximum $\bx^* \in \lexmax X$ as $c \goes \infty$. Theorem~\ref{thm:upper} states that the smallest and second smallest components of $\bx_{c, \gamma}$ are never more than
$O(1/c)$ below
their lexicographically maximum values, provided that $\gamma \in [0, 1)$,
so that an arbitrarily good approximation is possible by making $c$ large.
Note that the theorem makes no assumptions about $X$, not even that it is lexicographically stable. While the rate of convergence for the smallest component of $\bx_{c, \gamma}$ has been studied previously \citep{Rosset2004}, we believe that we are the first to prove unconditional convergence, even asymptotically, for the second smallest component.

\begin{theorem} \label{thm:upper} For all $n \ge 2$, $X \subseteq \bbR^n, \bx^* \in \lexmax X, c > 0$ and $\gamma \in [0, 1)$ 
\[
d_{\{1, 2\}}\barpair{\bx^*}{\bx_{c, \gamma}} \le \frac1c \log\left(\frac{n - k + 1}{1 - \gamma}\right).
\]
\end{theorem}
\begin{proof} By Lemma~\ref{thm:helper} we only need to show 
\[
\sup_{\bx \in X_{k-1}(\bx_{c,\gamma})} \sigma_k(\bx) \ge \sigma_k(\bx^*) 
\]
for $k \in \{1, 2\}$. By the definition of $\bx^*$ we have
$\sup_{\bx \in X} \sigma_1(\bx) = \sigma_1(\bx^*)$,
and since $X_0(\bx_{c, \gamma}) = X$ this implies $\sup_{\bx \in X_0(\bx_{c,\gamma})} \sigma_1(\bx) = \sigma_1(\bx^*)$.
We also have
\begin{align*}
\sigma_2(\bx^*) &= \sup_{\bx \in X : \sigma_1(\bx) \ge \sigma_1(\bx^*)} \sigma_2(\bx) & \because \textrm{Definition of }\bx^*\\
&\le \sup_{\bx \in X : \sigma_1(\bx) \ge \sigma_1(\bx_{c,\gamma})} \sigma_2(\bx) & \because \sigma_1(\bx^*) \ge \sigma_1(\bx_{c,\gamma}) \\
&= \sup_{\bx \in X_1(\bx_{c,\gamma})} \sigma_2(\bx) & \because \textrm{Eq.~\eqref{eq:constraints}} & \qedhere
\end{align*}
\end{proof}

In contrast to Theorem~\ref{thm:upper}, the situation is very different for the $k$th smallest component of $\bx_{c, \gamma}$ for all $k \ge 3$. Theorem~\ref{thm:linesegmentlowerbound} states that this component can remain far below its lexicographically maximum value for arbitrarily large values of $c$, even if $X$ is a bounded line segment (and thus is lexicographically stable) and $\gamma = 0$ ({i.e.}, the minimization is exact).

\begin{theorem} \label{thm:linesegmentlowerbound} For all $n \ge k \ge 3$ and $a \ge 1$ there exists a line segment $X \subseteq \bbR^n$ satisfying
$\norm{X}_\infty = 1$ such that for all $\bx^* \in \lexmax X$ and $c > 0$
\[
d\barpair{\bx^*}{\bx_{c,0}} \ge d_{\{k\}}\barpair{\bx^*}{\bx_{c,0}} \ge \frac13 \min\left\{1, \frac{a}{c}\right\}.
\]
\end{theorem}
\begin{proof}[Proof sketch] We consider only the case $n = 3$ here, as the general case $n \ge 3$ proceeds very similarly. The complete proof is provided in
Appendix~\ref{sec:linesegmentlowerbound}.

If $n = 3$ then we define $X$ to be the line segment joining the following two points:
\[
\bx^* = \left(\eps, \eps, 1\right)^\top \textrm{ and }
\bx' = \left(0, \frac23, \frac23\right)^\top
\]
where $\eps > 0$. Clearly $\bx^*$ is the lexicographic maximum of $X$.
As discussed earlier,
if $c$ is large then the dominant term in $L_c(\bx)$ corresponds to the smallest component of $\bx$. However if $c$ is small then several of the largest terms in $L_c(\bx)$ can have similar magnitude. We show that if $c \le a = \Omega(\log \frac1\eps)$ then at least the two largest terms in $L_c(\bx)$, which correspond to the two smallest components of $\bx$, have roughly the same magnitude. In this case the minimizer of $L_c(\bx)$ will be much closer to $\bx'$ than to $\bx^*$, because the second smallest component of $\bx'$ is roughly $\frac23$ larger than the second smallest component of $\bx^*$, while the smallest component of $\bx^*$ is only $\eps$ larger than the smallest component of $\bx'$.
\end{proof}

\subsection{A related algorithm using multiplicative weights}

Next, we discuss a related approach for finding a lexicographic maximum using no-regret strategies to solve an associated zero-sum game,
as was considered in great detail by \cite{Syed2010}.
This is another natural approach for finding lexicographic maxima
since, at least when $X$ is a convex polytope, we can view the lexicographic maximization computational task through the lens of solving a zero-sum game, a problem where no-regret algorithms have found a great deal of use. 

In the game theory perspective, we are trying to solve the following minimax problem:
$$
    \min_{\bp \in \Delta_m} \max_{\bq \in \Delta_n} \bp^\top \M \bq
$$
where $\M \in \bbR^{m \times n}$ and $\Delta_m, \Delta_n$ are the
probability simplices on $m,n$ items, respectively. An
\textit{equilibrium pair} of this minimax problem is a pair of distributions $\hat \bp \in \Delta_m$ and $\hat \bq \in \Delta_n$ satisfying
\[
    \bp^\top \M \hat \bq  \geq \hat \bp^\top \M \hat \bq \geq \hat \bp^\top \M \bq  
\]
for all $\bp \in \Delta_m$ and $\bq \in \Delta_n$. 
\citet{v1928theorie} showed that such a ``minimax-optimal'' pair, commonly known as a Nash equilibrium in a zero-sum game, always exists for every $\M$. 
There has been considerable work on how to compute such a pair,
including through the use of no-regret online learning algorithms for
sequentially updating $\bp$ and~$\bq$. For example, Multiplicative Weights (Algorithm \ref{alg:mwu}) is known to compute an $\epsilon$-approximate equilibrium of the game given by $\M$, with
$\epsilon = O\parens{\sqrt{\log(m)/T} + \sqrt{\log(n)/T}}$
\citep{FreundSc99}.

\RestyleAlgo{ruled}
\SetKw{KwRet}{Return:}
\begin{algorithm}\caption{Multiplicative weights method for computing a Nash equilibrium \label{alg:mwu}}
\KwIn{$\M \in \bbR^{m \times n}$, num. iter. $T$}
\KwIn{$\eta_1, \eta_2, \ldots > 0$ learning parameters}
$\bp_1 \gets \parens{\frac 1 n, \ldots, \frac 1 n}^\top$

\For{$t = 1, \ldots, T$}{
$\bq_t \leftarrow \text{ any element in } \arg\max_{\bq \in \Delta_n} \bp_t^\top \M \bq$\\
$\bp_{t+1} \leftarrow \frac 1 {Z_{t+1}} \exp\left(-\eta_t \M \sum_{s=1}^{t} \bq_s\right) 
$\\
\quad \quad where $\exp$ is applied coordinate wise, \\
\quad \quad and $Z_{t+1}$ is the normalizer.
}
\KwRet{$\bar \bp_T := \frac 1 T \sum_{t=1}^T \bp_t,\bar \bq_T := \frac 1 T \sum_{t=1}^T \bq_t$}
\end{algorithm}

In this paper, we study how to compute $\lexmax X$ for some $X \subseteq \bbR^m$.
Let us suppose $X$ is a convex polytope, that is, the convex hull of a
finite set of points $S\subseteq\Rm$.
Let $\M$ be a matrix whose columns are the points in $S$, so that
$X = \{\M\bq : \bq \in \Delta_n\}$.
We can then define a lexmax equilibrium strategy for the column player as any $\bq^*$ for which $\M\bq^* \in \lexmax X$. 
A ``lexmin'' equilibrium strategy for the row player
can be defined similarly.
Thus, a lexmin and lexmax equilibrium pair $\bp^*, \bq^*$ is a ``special'' Nash equilibrium which satisfies the additional constraints of being lexicographically optimal.
In this work, we focus only on computing a lexmax solution $\bq^*$ for
the column player.

The use of no-regret algorithms (such as Multiplicative Weights) has been very helpful for finding Nash equilibria in zero-sum games, among many other applications, and the exponentiation used in the update in Algorithm~\ref{alg:mwu} has an attractive similarity to the minimization of the exponential loss (\eqnref{eq:exploss}) considered primarily in this work. A natural question is whether Algorithm~\ref{alg:mwu} is suitable for finding not just any equilibrium strategy, but a lexmax equlibrium strategy $\bp^*$ as defined above.
Unfortunately, prior work suggests this is not the case
when the learning parameter $\eta_t$ is fixed to a constant:

\begin{theorem}[Informal summary of \protect{\cite[Theorem~3.7]{Syed2010}}]
There is a family of game matrices $\M \in \bbR^{3 \times 4}$ such that if Algorithm~\ref{alg:mwu} is run with a constant learning parameter $\eta_t = \eta$, the output $\bar \bq_T$ will not converge to a lexmax equilibrium strategy for the row player as $T \to \infty$.
\end{theorem}

On the positive side, \citet{Syed2010} also gives a result for when Algorithm~\ref{alg:mwu} computes a lexocographically optimal solution, but only in a very specific case where the solution has distinct values.

We now aim to rehabilitate Algorithm~\ref{alg:mwu}, the work of \cite{Syed2010} notwithstanding. As we will see, the choice of learning parameters $\eta_t$ is indeed quite important. First, let us define the function $H_c : \Delta_n \to \bbR$ as
\begin{align}
    H_c(\bq) := \frac 1 c \log \left( \sum_{i=1}^m \exp(-c \be_i^\top M \bq) \right),
\end{align}
where $\be_i$ is the $i$th basis vector.
This function is strongly related to $L_c(\cdot)$ as in \eqnref{eq:exploss}, except that we have $\frac 1 c \log(\cdot)$ operating on the outside. Notice, however, that the log transformation is monotonically increasing, so any minimizer of the exponential loss also minimizes $H_c(\cdot)$. Second, we observe that $H_c$ is \textit{$c$-smooth} --- that is, it satisfies $\| \nabla H_c(\bq) - \nabla H_c(\bq') \| \leq c\|\bq -\bq'\|$ for any $\bq, \bq' \in \Delta_n$.

To give our main result in this section, we emphasize that the
following leans on a  primal-dual perspective on optimization
that uses aforementioned tools on game playing.
See \cite{wang2023no} for a complete description.
\begin{theorem}
    If Algorithm~\ref{alg:mwu} is run with parameter $\eta_t = \frac c t$ then 
    \begin{equation*}
        H_c(\bar \bq_T) - \min_{\bq \in \Delta_n} H_c(\bq) = O\left(\frac{c \log T}{T}\right).
    \end{equation*}
\end{theorem}
\begin{proof}
This proof proceeds in three parts. First, we recall the well-known
\textit{Frank-Wolfe} procedure for minimizing smooth convex functions
on constrained sets,  applied here to $H_c(\bq)$.
Second, we show that, with the appropriate choice of update parameters, Frank-Wolfe applied to $H_c(\bq)$ is identical Algorithm~\ref{alg:mwu}. Finally, we appeal to standard convergence guarantees for Frank-Wolfe to obtain the desired convergence rate.
    
The Frank-Wolfe algorithm, applied to $H_c(\cdot)$, is as follows.
Let  $\bar \bq_0 \in \Delta_n$ be an arbitrary initial point,
and let $\gamma_1, \gamma_2, \ldots > 0$ be a step-size schedule.
On each iteration $t=1, \ldots, T$, we perform the following:
    \begin{align}
        \nabla_t & \gets \nabla H_c(\bar \bq_{t-1}) \\
        \bq_t & \gets \arg\min_{\bq \in \Delta_n} \langle \bq, \nabla_t \rangle \\
        \bar \bq_t & \gets (1-\gamma_t) \bar \bq_{t-1} + \gamma_t \bq_t.
    \end{align}
Ultimately, the algorithm returns $\bar \bq_T$. 

    We now show that this implementation of Frank-Wolfe is identical to Algorithm~\ref{alg:mwu}, as long as we have $\gamma_t = \frac 1 t$. To see this, we need to observe that the gradients $\nabla_t$ can be written as
    \begin{align*}
        \nabla_t = \nabla H_c(\bar \bq_{t-1}) 
        & \propto -\exp(-c \M \bar \bq_{t-1}) \\
        & = -\exp\left( - \frac c {t-1} \M\sum_{s=1}^{t-1} \bq_s \right).
    \end{align*}
    In other words, on every round we have that $\bp_t = - \nabla_t$. Furthermore, the vectors $\bq_t$ are chosen in exactly the same way, as the difference in sign is accounted for by the fact that $\bq_t$ is chosen as an $\arg\max$ in Algorithm~\ref{alg:mwu} as opposed to an $\arg\min$ in Frank-Wolfe.

    Let us finally recall a result that can be found in \cite{abernethy2017frank}, that the Frank-Wolfe algorithm run with parameters $\gamma_t = \frac 1 t$, for any function which is $\alpha$-smooth on its domain, converges at a rate of $O(\alpha \log(T)/T)$. We can now appeal to a well-known result on the convergence of Frank-Wolfe.\footnote{We emphasize that the more classical version of Frank-Wolfe uses slightly different update parameters,  $\gamma_t = \frac 2 {t+2}$, and this indeed can be used to obtain convergence that removes the $\log T$ dependence. But this is outside of the scope of this work.}
\end{proof}

This observation emphasizes that the popular multiplicative update method (Algorithm~\ref{alg:mwu}), that appears to fail to solve the desired problem according to the work of \cite{Syed2010}, actually succeeds in finding a lexicographic maximum insofar as the exponential loss minimization scheme succeeds. The critical ``patch'' that fixes the algorithm is the modified learning rate of $\frac c t$ in place of fixed $\eta$.

\section{Conclusion}

We proved a close connection between the two primary methods for
computing a lexicographic maximum of a set, and used this connection
to show that the method based on exponential loss minimization
converges to a correct solution for sets that are lexicographically
stable.
We believe our results represent the most general convergence criteria for exponential loss minimization that are known. We also undertook the first analysis of the convergence rate of exponential loss minimization, and found that even when convergence is guaranteed, the components of the minimizing vector can converge at vastly different rates. Finally, we showed that the well-known Multiplicative Weights algorithm can find a lexicographic maximum of a lexicographically stable set if the learning rate is suitably chosen.

\section*{Acknowledgements}

We thank Miroslav Dud\'ik , Saharon Rosset and Matus Telgarsky and for helpful discussions, and the anonymous reviewers for their comments and suggestions.

\bibliography{papers}

\begin{thebibliography}{}

\bibitem[Abernethy and Wang, 2017]{abernethy2017frank}
Abernethy, J.~D. and Wang, J.-K. (2017).
\newblock On frank-wolfe and equilibrium computation.
\newblock {\em Advances in Neural Information Processing Systems}, 30.

\bibitem[Bei et~al., 2022]{bei2022truthful}
Bei, X., Lu, X., and Suksompong, W. (2022).
\newblock Truthful cake sharing.
\newblock In {\em Proceedings of the AAAI Conference on Artificial
  Intelligence}, volume~36, pages 4809--4817.

\bibitem[Bertsekas and Gallager, 2021]{bertsekas2021data}
Bertsekas, D. and Gallager, R. (2021).
\newblock {\em Data networks}.
\newblock Athena Scientific.

\bibitem[Diana et~al., 2021]{diana2021lexicographically}
Diana, E., Gill, W., Globus-Harris, I., Kearns, M., Roth, A., and
  Sharifi-Malvajerdi, S. (2021).
\newblock Lexicographically fair learning: Algorithms and generalization.
\newblock In {\em 2nd Symposium on Foundations of Responsible Computing},
  page~1.

\bibitem[Dresher, 1961]{dresher1961mathematics}
Dresher, M. (1961).
\newblock The mathematics of games of strategy: Theory and applications
  prentice-hall.
\newblock {\em Englewood Cliffs, NJ}.

\bibitem[Freund and Schapire, 1999]{FreundSc99}
Freund, Y. and Schapire, R.~E. (1999).
\newblock Adaptive game playing using multiplicative weights.
\newblock {\em Games and Economic Behavior}, 29:79--103.

\bibitem[Hartman et~al., 2023]{hartman2023leximin}
Hartman, E., Hassidim, A., Aumann, Y., and Segal-Halevi, E. (2023).
\newblock Leximin approximation: From single-objective to multi-objective.

\bibitem[Hayden, 1981]{hayden1981voice}
Hayden, H.~P. (1981).
\newblock {\em Voice flow control in integrated packet networks}.
\newblock PhD thesis, Massachusetts Institute of Technology.

\bibitem[Henzinger et~al., 2022]{henzinger2022leximax}
Henzinger, M., Peale, C., Reingold, O., and Shen, J.~H. (2022).
\newblock Leximax approximations and representative cohort selection.
\newblock {\em arXiv preprint arXiv:2205.01157}.

\bibitem[Le~Boudec, 2000]{LeBoudec2000}
Le~Boudec, J.-Y. (2000).
\newblock Rate adaptation, congestion control and fairness: A tutorial.
\newblock {\em Ecole Polytechnique Federale de Lausanne}.

\bibitem[Littlestone and Warmuth, 1994]{littlestone1994weighted}
Littlestone, N. and Warmuth, M.~K. (1994).
\newblock The weighted majority algorithm.
\newblock {\em Information and computation}, 108(2):212--261.

\bibitem[Luss, 1999]{luss1999equitable}
Luss, H. (1999).
\newblock On equitable resource allocation problems: A lexicographic minimax
  approach.
\newblock {\em Operations Research}, 47(3):361--378.

\bibitem[Luss and Smith, 1986]{luss1986resource}
Luss, H. and Smith, D.~R. (1986).
\newblock Resource allocation among competing activities: A lexicographic
  minimax approach.
\newblock {\em Operations Research Letters}, 5(5):227--231.

\bibitem[Mart{\'\i}nez-Legaz and Singer, 1987]{martinez1987lexicographical}
Mart{\'\i}nez-Legaz, J.-E. and Singer, I. (1987).
\newblock Lexicographical separation in rn.
\newblock {\em Linear Algebra and Its Applications}, 90:147--163.

\bibitem[Miltersen and S{\o}rensen, 2006]{miltersen2006computing}
Miltersen, P.~B. and S{\o}rensen, T.~B. (2006).
\newblock Computing proper equilibria of zero-sum games.
\newblock In {\em International Conference on Computers and Games}, pages
  200--211. Springer.

\bibitem[Mo and Walrand, 1998]{MoWalrand98}
Mo, J. and Walrand, J. (1998).
\newblock {Fair end-to-end window-based congestion control}.
\newblock In Lai, W.~S. and Cooper, R.~B., editors, {\em Performance and
  Control of Network Systems II}, volume 3530, pages 55 -- 63. International
  Society for Optics and Photonics, SPIE.

\bibitem[Nacson et~al., 2019]{Nacson2019}
Nacson, M.~S., Gunasekar, S., Lee, J., Srebro, N., and Soudry, D. (2019).
\newblock Lexicographic and depth-sensitive margins in homogeneous and
  non-homogeneous deep models.
\newblock In {\em International Conference on Machine Learning}, pages
  4683--4692. PMLR.

\bibitem[Rockafellar, 1970]{ROC}
Rockafellar, R.~T. (1970).
\newblock {\em Convex Analysis}.
\newblock Princeton University Press.

\bibitem[Rosset et~al., 2004]{Rosset2004}
Rosset, S., Zhu, J., and Hastie, T. (2004).
\newblock Boosting as a regularized path to a maximum margin classifier.
\newblock {\em J. Mach. Learn. Res.}, 5:941–973.

\bibitem[Schmeidler, 1969]{schmeidler1969nucleolus}
Schmeidler, D. (1969).
\newblock The nucleolus of a characteristic function game.
\newblock {\em SIAM Journal on applied mathematics}, 17(6):1163--1170.

\bibitem[Sherali and Soyster, 1983]{sherali1983preemptive}
Sherali, H.~D. and Soyster, A.~L. (1983).
\newblock Preemptive and nonpreemptive multi-objective programming:
  Relationship and counterexamples.
\newblock {\em Journal of Optimization Theory and Applications}, 39:173--186.

\bibitem[Syed, 2010]{Syed2010}
Syed, U.~A. (2010).
\newblock {\em Reinforcement learning without rewards}.
\newblock Ph.D. Thesis.

\bibitem[Van~Damme, 1991]{van1991stability}
Van~Damme, E. (1991).
\newblock {\em Stability and perfection of Nash equilibria}, volume 339.
\newblock Springer.

\bibitem[Von~Neumann, 1928]{v1928theorie}
Von~Neumann, J. (1928).
\newblock Zur theorie der gesellschaftsspiele.
\newblock {\em Mathematische annalen}, 100(1):295--320.

\bibitem[Wang et~al., 2023]{wang2023no}
Wang, J.-K., Abernethy, J., and Levy, K.~Y. (2023).
\newblock No-regret dynamics in the fenchel game: A unified framework for
  algorithmic convex optimization.
\newblock {\em Mathematical Programming}, pages 1--66.

\end{thebibliography}

\appendix

\section{Appendix}
\label{sec:proofs}

\subsection{Lipschitz continuity of sorting functions}
\label{sec:sorting-cont}

\begin{theorem}\label{thm:continuitysorting} For all $k \in [n]$ and $\bx, \bx' \in \bbR^n$ we have $|\sigma_k(\bx) - \sigma_k(\bx')| \le 3\norm{\bx - \bx'}_\infty$.
Therefore, $\sigma_k$ is continuous.
\end{theorem}
\begin{proof} Let $\eps = \norm{\bx - \bx'}_\infty$ and assume without loss of generality that $\sigma_k(\bx) \le \sigma_k(\bx')$. It suffices to show that $\sigma_k(\bx') \le \sigma_k(\bx) + 3\eps$. Choose $i, j \in [n]$ such that $\sigma_k(\bx) = x_i$ and $\sigma_k(\bx') = x'_j$. If $x_i \ge x_j - 2\eps$ then
\[
\sigma_k(\bx') = x'_j \le x_j + \eps \le x_i + 3\eps = \sigma_k(\bx) + 3\eps.
\]
It remains to show that $x_i < x_j - 2\eps$ cannot be true. Suppose it is. Since $\sigma_k(\bx) = x_i$, we have
\[
k \le |\{ \ell \in [n] : x_\ell \le x_i\}| \le |\{ \ell \in [n] : x_\ell < x_j - 2\eps\}|.
\]
Since $\max_{\ell \in [n]} |x_\ell - x'_\ell| \le \eps$, this implies
\[
k \le |\{ \ell \in [n] : x'_\ell < x'_j\}|
\]
which implies $\sigma_k(\bx') < x'_j$, a contradiction.
\end{proof}

\subsection{Proof of Theorem \ref{thm:convex}}
\label{sec:convex-polytop-stable}

In this section, we prove Theorem~\ref{thm:convex}.
We first sketch an outline.

As preliminary steps, we begin by introducing a function $\Hk$
that, in a sense, summarizes the optimization problem being solved
(approximately) on each round of Algorithm~\ref{alg:iterative}.
Since this function is somewhat difficult to work with, we next
introduce another function with much more favorable properties,
denoted $\GI$, where $I\subseteq [n]$.
We show that, if $|I|=k-1$ then $\Hk\geq\GI$.
Moreover, if $X$ is a convex polytope, then
$\GI$ is concave and lower semicontinuous.

With these definitions and preliminaries, we then proceed below with the
proof of Theorem~\ref{thm:convex}.
By way of contradiction, we suppose that $X$ is a convex polytope but
nevertheless is not lexicographically stable.
This implies the existence of a sequence $\eps_t>0$ with
$\eps_t\rightarrow 0$, and another sequence $\xx_t\in\algout{\eps_t}$
converging to a point $\xhat\in X$ that is different than $X$'s
lexicographic maximum $\xopt$.
With some re-indexing,
we then identify an index $k\in [n]$ such that $\Hk(\xx_t)$ is
asymptotically upper-bounded by $\hatx_k$, while, on the other hand
$\GI(\xhat)$ strictly exceeds $\hatx_k$, for some $I\subseteq [n]$
with $|I|=k-1$.
Combining, these facts lead to the contradiction:
\[
  \hatx_k
  <
  \GI(\xhat)
  \leq
  \liminf_{t\rightarrow\infty} \GI(\xx_t)
  \leq
  \liminf_{t\rightarrow\infty} \Hk(\xx_t)
  \leq
  \hatx_k.
\]

We now provide details.
Let $\Rext=\R\cup\{-\infty,+\infty\}$.
For $k\in[n]$, we define the function
$\Hk:\Rn\rightarrow\Rext$ by
\begin{equation}  \label{eqn:hk-defn}
  \Hk(\xx)
  =
  \sup\braces{ \sigma_k(\zz) :
               \zz\in X, \sigma_i(\zz)\geq\sigma_i(\xx)
                         \mbox{ for } i=1,\ldots,k-1 }.
\end{equation}
Then Algorithm~\ref{alg:iterative} operates exactly by choosing, on each round
$k$, $\xs{k}$ so that
\begin{equation}  \label{eqn:hk-rel-alg}
  \sigma_k(\xs{k})\geq \Hk(\xs{k-1})-\eps
\end{equation}
(and also so that $\sigma_i(\xs{k})\geq\sigma_i(\xs{k-1})$ for
$i\in [k-1]$).

In fact, this condition can be simplified:
As we show next, a point $\xx$ is a possible output
of Algorithm~\ref{alg:iterative} if and only if it satisfies
\eqnref{eqn:hk-rel-alg} with both 
$\xs{k-1}$ and $\xs{k}$ replaced by $\xx$.
Said differently, $\xx$ is valid in the sense of the proof of
Theorem~\ref{thm:main}
if and only if it is a possible output.

\begin{proposition}  \label{pr:poss-output-simplify}
  Let $\xx\in X$, and let $\eps\geq 0$.
  Then $\xx\in\algout{\eps}$ if and only if
  \begin{equation}  \label{eq:pr:poss-output-simplify:1}
    \sigma_k(\xx)\geq \Hk(\xx) -\eps
  \end{equation}
  for all $k\in [n]$.
\end{proposition}

\begin{proof}
If \eqnref{eq:pr:poss-output-simplify:1} holds for all $k\in [n]$,
then, in running Algorithm~\ref{alg:iterative}, we can choose
$\xs{k}=\xx$ on every iteration, proving that
$\xx\in\algout{\eps}$.

Conversely, suppose $\xx\in\algout{\eps}$.
Let $\xs{k}$, for $k\in [n]$, be the sequence of iterates computed by
Algorithm~\ref{alg:iterative} resulting in the final output
$\xs{n}=\xx$.
Then, by the manner in which these are computed,
$\sigma_i(\xx)=\sigma_i(\xs{n})\geq\cdots\geq\sigma_i(\xs{i})$
for $i\in [n]$.
This implies, for $k\in [n]$, that
$\Hk(\xx)\leq \Hk(\xs{k-1})$ since if
$\sigma_i(\zz)\geq \sigma_i(\xx)$ then
$\sigma_i(\zz)\geq \sigma_i(\xs{k-1})$ for $i<k$.
Thus,
\[
  \sigma_k(\xx)
  \geq
  \sigma_k(\xs{k})
  \geq
  \Hk(\xs{k-1})-\eps
  \geq
  \Hk(\xx)-\eps,
\]
proving
\eqnref{eq:pr:poss-output-simplify:1}.
\end{proof}

Next, let $I\subsetneq [n]$ be a set of indices
(other than $[n]$), and
let us define the functions
$\fI:\Rn\rightarrow\R$ and $\GI:\Rn\rightarrow\Rext$ by
\[
  \fI(\xx)
  =
  \min_{i\in [n]\setminus I} x_i
\]
and
\begin{equation} \label{eqn:gi-defn}
  \GI(\xx)
  =
  \sup\braces{ \fI(\zz) :
                 \zz\in X,
                 z_i = x_i \mbox{ for } i\in I
             }
\end{equation}
for $\xx\in\Rn$.
Thus, $\fI(\xx)$ is equal to the minimum value of the components not
in $I$, and
$\GI(\xx)$ is the maximum value of $\fI(\zz)$ over all points
$\zz\in X$ that agree with $\xx$ on all components in $I$.
The function $\GI$ bears some resemblence to $\Hk$, but, as we will
see, is easier to work with since, with the set $I$ fixed, we avoid
the sorting functions $\sigma_i$.
The next proposition establishes that connection:

\begin{proposition}  \label{pr:hk-le-gi}
  Let $k\in [n]$ and
  let $I\subseteq [n]$ with $|I|=k-1$.
  Then
  $\Hk(\xx)\geq \GI(\xx)$
  for all $\xx\in\Rn$.
\end{proposition}  

\begin{proof}
Let $\xx\in \Rn$, and
let $\zz\in X$ be such that $z_i=x_i$ for $i\in I$.
We show, in cases,
that $\fI(\zz)\leq \Hk(\xx)$.
From $\GI$'s definition (\eqnref{eqn:gi-defn}), this will prove the
proposition.

Suppose, in the first case, that $\sigma_i(\zz)<\sigma_i(\xx)$ for
some $i\in I$.
There must exist a set of indices $J\subseteq [n]$ such that
$|J|=i$ and
\begin{equation}  \label{eq:pr:hk-le-gi:1}
  \sigma_i(\zz)=\max\{z_j : j\in J\}.
\end{equation}
We claim $J\not\subseteq I$.
Otherwise, if $J\subseteq I$, then we must have
\[
  \sigma_i(\zz)
  =
  \max\{x_j : j\in J\}
  \geq
  \sigma_i(\xx).
\]  
The equality is by \eqnref{eq:pr:hk-le-gi:1} and
since $z_j=x_j$ for $j\in I$.
The inequality is because $\sigma_i(\xx)$ is the
$i$-th smallest component of $\xx$ and $|J|=i$.
This contradicts that $\sigma_i(\zz)<\sigma_i(\xx)$.

Thus, there must exist $j\in J\setminus I$, implying
\[
  \fI(\zz)
  \leq
  z_j
  \leq
  \sigma_i(\zz)
  <
  \sigma_i(\xx)
  \leq
  \sigma_k(\xx)
  \leq
  \Hk(\xx).
\]
The first inequality is because $j\not\in I$;
the second is by \eqnref{eq:pr:hk-le-gi:1} and since $j\in J$;
the third is by assumption;
the fourth is because $i<k$;
and the last is by $\Hk$'s definition
(\eqnref{eqn:hk-defn}).

In the alternative case, $\sigma_i(\zz)\geq \sigma_i(\xx)$
for all $i\in I$, implying that $\zz$ satisfies the conditions
appearing in $\Hk$'s definition.
Then, similar to the preceding arguments, there must exist a set
$J\subseteq [n]$ with $|J|=k$ and such that
$\sigma_k(\zz) = \max\{z_j : j\in J\}$.
Since $|I|=k-1$, this implies that there must exist
$j\in J\setminus I$.
Thus,
$\fI(\zz) \leq z_j \leq \sigma_k(\zz) \leq \Hk(\xx)$
where the last inequality follows from $\Hk$'s definition.
\end{proof}

A function $f:\Rn\rightarrow\Rext$ is
\emph{lower semicontinuous relative to a set $S\subseteq\Rn$}
if for all $\xx\in S$ and for every sequence $\xx_t$ in $S$
with $\xx_t\rightarrow\xx$,
we have
\[ \liminf_{t\rightarrow\infty} f(\xx_t) \geq f(\xx). \]
Similarly,
the function is
\emph{upper semicontinuous relative to $S$}
if instead
$ \limsup_{t\rightarrow\infty} f(\xx_t) \leq f(\xx) $
whenever $\xx_t\rightarrow\xx$
(with $\xx_t\in S$).

We next prove useful properties of $\GI$ when $X$ is a convex
polytope:

\begin{lemma}  \label{lem:gi-conc-lsc}
  Let $I\subsetneq [n]$, and assume $X$ is convex and nonempty.
  Then $\GI$ is concave.
  If, in addition, $X$ is a convex polytope, then
  $\GI$ is lower semicontinuous relative to $X$.
\end{lemma}

\begin{proof}
Let $k=|I|$ and suppose $I=\{i_1,\ldots,i_k\}$.
Let $P:\Rn\rightarrow\Rk$ be the linear mapping that projects a point
$\xx$ onto just the coordinates in $I$
(so that $[P(\xx)]_j=x_{i_j}$ for $\xx\in\Rn$ and $j\in [k]$).
Let $\indX$ be the indicator function for $X$ so that
$\indX(\xx)$ is $0$ if $\xx\in X$ and is $+\infty$ otherwise.
Let $g:\Rk\rightarrow\Rext$ be defined by
\[
  g(\yy)
  =
  \inf\braces{ -\fI(\zz) + \indX(\zz) : \zz\in\Rn, P(\zz) = \yy }
\]
for $\yy\in \Rk$.
Note that $-\fI$ is convex, being the pointwise maximum of linear (and
so convex) functions, and $\indX$ is also convex since $X$ is.
Therefore $-\fI + \indX$ is convex.
It follows that $g$ is convex, being the so-called image of
$-\fI+\indX$ under $P$.
Note further that $-\GI(\xx)=g(P(\xx))$ for $\xx\in\Rn$, so
$-\GI$ is
convex as well, proving $\GI$ is concave.
(In showing $g$ and $-\GI$ are convex, we applied general facts given
in \citet[Theorem~5.7]{ROC}.)

If $\xx\in X$, then, from $\GI$'s definition,
$-\GI(\xx)\leq -\fI(\xx)<+\infty$;
thus, $X$ is included in $-\GI$'s effective domain (set of points
where it is not $+\infty$).
If $X$ is a convex polytope, then it is also
\emph{locally simplicial}, and, being included in $-\GI$'s effective
domain, it then follows that $-\GI$ is upper
semicontinuous relative to $X$
\citep[Theorems~10.2 and~20.5]{ROC}.
Therefore, $\GI$ is lower semicontinuous relative to $X$.
\end{proof}

\begin{proof}[Proof of Theorem \ref{thm:convex}]
Let $X\subseteq \Rn$ be a convex polytope.
Let $\xopt$ be the unique lexicographic maximum
(which exists by Theorem~\ref{thm:lexopt-unique}).
Suppose by way of contradiction that $X$ is not lexicographically
stable.
Then there exists $\delta>0$ such that for all $\eps>0$
there exists $\xhat\in\algout{\eps}$ with
$d\barpair{\bx^*}{\hbx} \geq \delta$, implying
that
$\sigma_{k}(\xhat)\leq \sigma_{k}(\xopt) - \delta$
for some $k\in [n]$.
Since there are only finitely many values in $[n]$, this in turn
implies that
there exists $k_0\in [n]$, $\delta>0$,
a sequence $\eps_t>0$ with $\eps_t\rightarrow 0$,
and
a sequence $\xx_t\in\algout{\eps_t}$
such that
$\sigma_{k_0}(\xx_t)\leq \sigma_{k_0}(\xopt) - \delta$
for infinitely many $t$.
By discarding all other elements from the sequence, we assume
henceforth that this holds for all $t$.

Since $X$ is a convex polytope, it is also compact.
Therefore, the sequence $\xx_t$ must have a subsequence converging to
some point $\xhat\in X$.
By discarding all other elements, we can assume the entire sequence
converges so that $\xx_t\rightarrow\xhat$.
Further, by continuity
(Theorem~\ref{thm:continuitysorting}),
$\sigma_{k_0}(\xhat) \leq \sigma_{k_0}(\xopt) - \delta$.
In particular, this shows that $\xhat\neq \xopt$.

By possibly permuting the components of points in $X$, we assume
without loss of generality that
the components are sorted according to the components of $\xhat$, and,
when there are ties, according to the components of $\xopt$.
That is, we assume the indices have been permuted in such a way
that for all $i,j\in [n]$, if $i\leq j$ then
$\hatx_i \leq \hatx_j$, and in addition,
if $\hatx_i = \hatx_j$ then $\xopti_i \leq \xopti_j$.
In particular, this implies
$\hatx_1\leq\cdots\leq\hatx_n$,
and $\sigma_i(\xhat)=\hatx_i$ for $i\in [n]$.

Let $k$ be the smallest index on which $\xhat$ and $\xopt$ differ
(so that $\hatx_i=\xopti_i$ for $i<k$ and $\hatx_k\neq\xopti_k$).
Let $I=[k-1]$.

\begin{claimpx}  \label{cl:thm:convex:1}
  $\sigma_i(\xopt)=\sigma_i(\xhat)$ for $i\in I$.
\end{claimpx}

\begin{proofx}
We prove, by induction on $m=0,1,\ldots,k-1$, that
$\sigma_i(\xopt)=\sigma_i(\xhat)$ for $i\leq m$.
This holds vacuously in the base case that $m=0$.
Suppose $m\in [k-1]$ and that the claim holds for $m-1$.
Then
$\sigma_i(\xopt)=\sigma_i(\xhat)$ for $i\leq m-1$.
Therefore, $\sigma_m(\xopt)\geq\sigma_m(\xhat)$
since $\xopt$ is lexicographically optimal.
On the other hand, since $\sigma_m(\xopt)$ is the $m$-th smallest
component of $\xopt$,
$\sigma_m(\xopt)\leq \max\{\xopti_1,\ldots,\xopti_m\}= \max\{\hatx_1,\ldots,\hatx_m\}=\hatx_m=\sigma_m(\xhat)$.
This completes the induction.
\end{proofx}

By Claim~\ref{cl:thm:convex:1},
$\sigma_i(\xopt)=\sigma_i(\xhat)=\hatx_i=\xopti_i$ for
$i\in I$.
Therefore, $\sigma_k(\xopt)$ is the smallest of the remaining
components of $\xopt$.
Thus,
\begin{equation}  \label{eq:thm:convex:1}
  \xopti_k
  \geq
  \fI(\xopt)
  =
  \sigma_k(\xopt)
  \geq
  \sigma_k(\xhat)
  =
  \hatx_k
\end{equation}
where the second inequality follows from
Claim~\ref{cl:thm:convex:1}
since $\xopt$ is lexicographically maximal.
Since we assumed $\xopti_k\neq\hatx_k$, we must have
$\xopti_k > \hatx_k$.

\begin{claimpx}  \label{cl:thm:convex:2}
  $\fI(\xopt) > \hatx_k$.
\end{claimpx}

\begin{proofx}
The proof is very similar to the proof of
Theorem~\ref{thm:lexopt-unique}.
Suppose by way of contradiction that the claim is false.
Then, in light of \eqnref{eq:thm:convex:1}, we must have
$\sigma_k(\xopt)=\fI(\xopt)=\hatx_k$.
Let $\yy=(\xopt+\xhat)/2$,
which is in $X$ since $X$ is convex.
We compare the components of $\yy$ to $\hatx_k$.
Let $i\in [n]$.

If $i<k$ then $y_i = \hatx_i \leq \hatx_k$
(since $\hatx_i=\xopti_i$).

If $i=k$ then $y_k = (\xopti_k+\hatx_k)/2 >\hatx_k$
since $\xopti_k>\hatx_k$.

Finally, suppose $i>k$, implying $\hatx_i\geq\hatx_k$.
If $\hatx_i>\hatx_k$ then
$y_i=(\xopti_i+\hatx_i)/2>\hatx_k$ since
$\xopti_i\geq\hatx_k$
(by \eqnref{eq:thm:convex:1}).
Otherwise, $\hatx_i=\hatx_k$, implying, by how the components are
sorted, that
$\xopti_i\geq\xopti_k>\hatx_k$;
thus, again,
$y_i=(\xopti_i+\hatx_i)/2>\hatx_k$.

To summarize, $y_i=\xopti_i\leq \hatx_k$ if $i\in I$,
and $y_i>\hatx_k$ if $i\not\in I$.
It follows that
$\sigma_i(\yy)=\sigma_i(\xopt)=\xopti_i$ for $i=1,\ldots,k-1$,
and that
$\sigma_k(\yy)=\fI(\yy)>\hatx_k=\sigma_k(\xopt)$.
However, this contradicts that $\xopt$ is lexicographically maximal.
\end{proofx}

Combining, we now have
\[
  \hatx_k
  <
  \fI(\xopt)
  \leq
  \GI(\xhat)
  \leq
  \liminf_{t\rightarrow\infty} \GI(\xx_t)
  \leq
  \liminf_{t\rightarrow\infty} \Hk(\xx_t)
  \leq
  \liminf_{t\rightarrow\infty} [\sigma_k(\xx_t) + \eps_t]
  =
  \hatx_k.
\]
The first inequality is by Claim~\ref{cl:thm:convex:2}.
The second is from $\GI$'s definition
(\eqnref{eqn:gi-defn}).
The third is because $\GI$ is lower semicontinuous
relative to $X$
(Lemma~\ref{lem:gi-conc-lsc}).
The fourth is by
Proposition~\ref{pr:hk-le-gi}.
The fifth is by Proposition~\ref{pr:poss-output-simplify},
since
$\xx_t\in\algout{\eps_t}$.
The equality is because $\eps_t\rightarrow 0$ and $\xx_t\rightarrow\xhat$,
implying $\sigma_k(\xx_t)\rightarrow\sigma_k(\xhat)=\hatx_k$
(using Theorem~\ref{thm:continuitysorting}).

Having reached a contradiction, we conclude that $X$ is
lexicographically stable.
\end{proof}

\subsection{A compact, convex set that is not lexicographically stable}
\label{sec:convex-not-stable}

Let $X$ be as in \eqnref{eqn:compact-X-not-lex-stab}, which is convex and compact.
We show that $X$ is not lexicographically stable.

By Theorem~\ref{thm:lexopt-unique}, $X$ has a unique lexicographic
maximum $\xopt$.
We first argue that $\xopt=\transvec{0,0,1}$.
Note first that for all $\xx\in X$,
$x_1 \leq 0 \leq \min\{x_2,x_3\}$, implying
$\sigma_1(\xx)=x_1$.
In particular, since $\sigma_1(\xopt)$ maximizes $\sigma_1(\xx)$ over
$\xx\in X$, this means we must have $\xopti_1=0$.
By $X$'s definition, this implies $\xopti_2=0$.
We can then choose $\xopti_3=1$, since this is that component's
largest possible value.

Next, we argue that, for $\eps\in (0,1)$,
$\xeps=\transvec{-\eps^2,\eps/2,3/4}\in\algout{\eps}$,
and more specifically that $\xeps$ can be chosen for $\xs{k}$ on each
round $k$ of Algorithm~\ref{alg:iterative}.

On round~1,
as noted already, if $\xx\in X$ then $\sigma_1(\xx)=x_1\leq 0$.
Therefore,
$\sigma_1(\xeps) = -\eps^2$ is within $\eps$ of maximizing
$\sigma_1(\xx)$ over $\xx\in X$, so we can choose
$\xs{1}=\xeps$.

For round 2, suppose $\xx\in X_1(\xeps)$, where
$X_k(\xx)$ is as defined in
\eqnref{eq:constraints}.
That is, $x_1 = \sigma_1(\xx) \geq -\eps^2$.
Then
$x_2^2 \leq (1-x_3)(-x_1) \leq \eps^2$,
so $\sigma_2(\xx)\leq x_2\leq \eps$.
Therefore, $\sigma_2(\xeps) = \eps/2$ is within $\eps$ of maximizing
$\sigma_2(\xx)$ over $\xx\in X_1(\eps)$, so we can choose
$\xs{2}=\xeps$.

Finally, for round~3, suppose $\xx\in X_2(\xeps)$, implying
$x_1=\sigma_1(\xx)\geq -\eps^2$ and
$x_2 \geq \sigma_2(\xx)\geq \eps/2$.
These imply
\[
   \eps^2 (1-x_3)
   \geq
   (-x_1) (1-x_3)
   \geq
   x_2^2
   \geq
   (\eps/2)^2,
\]
and consequently that $x_3 \leq 3/4$.
Thus, $\sigma_3(\xeps) = 3/4$ maximizes $\sigma_3(\xx)$ over
$\xx\in X_2(\xeps)$.

We conclude that $\xeps\in\algout{\eps}$.
Since $\sigma_3(\xeps)=3/4$ for all $\eps\in (0,1)$
but $\sigma_3(\xopt)=1$
(implying $d\barpair{\xopt}{\xeps} \geq 1/4$),
$X$ is not lexicographically stable.

\subsection{Proof of Lemma~\ref{thm:helper}}
\label{sec:lem-helper}

\begin{proof} Fix $k \in [n]$ and choose any $\bx \in X_{k-1}(\bx_{c, \gamma})$.  By the definition of $\bx_{c, \gamma}$, the definition of $L_c$, and rearranging terms
\begin{align}
\nonumber & -\gamma \exp(-c\norm{X}_\infty) \le L_c(\bx) - L_c(\bx_{c, \gamma}) \\
= & \sum_{i = 1}^n \exp(-c\sigma_i(\bx)) - \exp(-c\sigma_i(\bx_{c, \gamma})). \label{eq:helper:one}
\end{align}
Since $\bx \in X_{k-1}(\bx_{c, \gamma})$ and the function $x \mapsto \exp(-x)$ is decreasing
\begin{align}
& \sum_{i = 1}^n \exp(-c\sigma_i(\bx)) - \exp(-c\sigma_i(\bx_{c, \gamma})) \nonumber \\
\le & \sum_{i = k}^n \exp(-c\sigma_i(\bx)) - \exp(-c\sigma_i(\bx_{c, \gamma})). \label{eq:helper:two}
\end{align}
Since the function $x \mapsto \exp(-x)$ is decreasing and positive
\begin{align}
& \sum_{i = k}^n \exp(-c\sigma_i(\bx)) - \exp(-c\sigma_i(\bx_{c, \gamma})) \nonumber \\
\le & (n - k + 1) \exp(-c\sigma_k(\bx)) - \exp(-c\sigma_k(\bx_{c, \gamma})). \label{eq:helper:three}
\end{align}
Combining \eqref{eq:helper:one}, \eqref{eq:helper:two}, and \eqref{eq:helper:three} and dividing through by $\exp(-c\sigma_k(\bx_{c, \gamma}))$ yields 
\begin{align*}
    & \frac{-\gamma\exp(-c\norm{X}_\infty)}{\exp(-c\sigma_k(\bx_{c, \gamma}))} \\
    \le & (n - k + 1)\exp(-c(\sigma_k(\bx) - \sigma_k(\bx_{c, \gamma}))) - 1.    
\end{align*}

Since $\sigma_k(\bx_{c, \gamma}) \le \norm{X}_\infty$ we have $\exp(-c\norm{X}_\infty) \le \exp(-c\sigma_k(\bx_{c, \gamma}))$ and therefore
\[
-\gamma \le (n - k + 1)\exp(-c(\sigma_k(\bx) - \sigma_k(\bx_{c, \gamma}))) - 1.
\]
By rearranging we have
\[
\sigma_k(\bx_{c, \gamma}) \ge \sigma_k(\bx)  - \frac1c \log\left(\frac{n - k + 1}{1 - \gamma}\right).
\]
Because $\bx \in X_{k-1}(\bx_{c, \gamma})$ was chosen arbitrarily this implies
\[
\sigma_k(\bx_{c, \gamma}) \ge \sup_{\bx \in X_{k-1}(\bx_{c, \gamma})} \sigma_k(\bx) - \frac1c \log\left(\frac{n - k + 1}{1 - \gamma}\right).
\] 
\end{proof}

\subsection{Proof of Theorem \ref{thm:sharp_lower}}
\label{sec:thm:sharp-lower}

Define $\bx^*, \bx', \bx'' \in \bbR^n$ as
\begin{align*}
    \bx^* &= \left(0, \ldots, 0, 1\right)^\top,\\
    \bx' &= \left(0, \ldots, 0, \frac12\right)^\top,\\
    \bx'' &= \left(-\frac12, \frac14, \ldots, \frac14, \frac12\right)^\top
\end{align*}
and let $X = \textrm{conv}(\{\bx^*, \bx'\}) \cup \textrm{conv}(\{\bx', \bx''\})$, where $\textrm{conv}(S)$ denotes the convex hull of $S \subseteq \bbR^n$. Clearly $\norm{X}_\infty = 1$. For all $\lambda \in [0, 1]$ let
\begin{align*}
    \bx(\lambda) &= \lambda \bx^* + (1 - \lambda) \bx',\\
    \bx'(\lambda) &= \lambda \bx' + (1 - \lambda) \bx''
\end{align*}
Since for all $\lambda \in [0, 1)$ 
\begin{align*}
    \sigma_k(\bx^*) &= \sigma_k(\bx(\lambda)) \textrm{ for } k \in [n-1] \textrm{ and }\sigma_n(\bx^*) > \sigma_n(\bx(\lambda)),\\
    \sigma_1(\bx^*) &> \sigma_1(\bx'(\lambda)).
\end{align*}
we have $\lexmax X = \{\bx^*\}$. We also know that $\bx_{c, 0}$ must exist, as it is defined to be the minimum of a continuous function on a compact set. Since for all $\lambda \in [0, 1]$
\begin{align*}
L_c(\bx^*) &\le L_c(\bx(\lambda))\\
\sigma_n(\bx'(\lambda)) &= \sigma_n(\bx^*) - \frac12
\end{align*}
it remains to show that for all $c \ge 2$ there exists $\lambda_c \in [0, 1]$ such that $L_c(\bx'(\lambda_c)) < L_c(\bx^*)$. Let $\lambda_c = 1 - \frac2c$. We have
\begin{align*}
L_c(\bx'(\lambda_c)) - L_c(\bx^*) &= \sum_{k=1}^n \exp(-c\sigma_k(\bx'(\lambda_c))) - \exp(-c\sigma_k(\bx^*))\\
&= \exp(1) - 1\\
&~~~~ + (n - 2)\exp\left(-\frac12\right) - (n - 2)\\
&~~~~ + \exp\left(-\frac{c}{2}\right) - \exp(-c)\\
&\le \exp(1) - 1\\
&~~~~ + (n - 2)\exp\left(-\frac12\right) - (n - 2)\\
&~~~~ + \exp\left(-1\right) & \because c \ge 2 \textrm{ and } \exp(-c) \ge 0\\
&< 0 & \because n \ge 8
\end{align*}

\subsection{Proof of Theorem \ref{thm:linesegmentlowerbound}}
\label{sec:linesegmentlowerbound}

Let $\beta = \frac23$. Choose $\eps \in (0, \frac18]$ such that $a = \frac1\beta \log \frac{2\beta - 1 - \eps}{\eps}$. This is feasible because the function $f(\eps) = \frac1\beta \log \frac{2\beta - 1 - \eps}{\eps}$ is decreasing and continuous on the interval $(0, \frac18]$, $\lim_{\eps \goes 0} f(\eps) = \infty$, $f(\frac18) < 1$ and $a \ge 1$. Note that $\eps < 2\beta - 1 < \beta$. Let
\begin{align*}
    \bx^* &= \bigg(\eps, \ldots, \eps, \hspace*{-0.35cm}\overset{\substack{k\textrm{th term}\\\downarrow}}{1}\hspace*{-0.35cm}, \ldots, 1\bigg)^\top \in \bbR^n,\\
    \bx' &= \left(0, \eps, \ldots, \eps, \beta, \hspace*{-0.35cm}\underset{\substack{\uparrow\\k\textrm{th term}}}{\beta}\hspace*{-0.35cm}, 1, \ldots, 1\right)^\top \in \bbR^n
\end{align*}
and $X = \textrm{conv}(\{\bx^*, \bx'\})$, where $\textrm{conv}(S)$ denotes the convex hull of $S \subseteq \bbR^n$. Clearly $\norm{X}_\infty = 1$. Let $\bx(\lambda) = \lambda \bx^* + (1 - \lambda) \bx'$ for $\lambda \in [0, 1]$. Observe that for $\lambda \in [0, 1)$
\[
\sigma_1(\bx^*) = \sigma_1(\bx(1)) = \eps > \lambda \eps = \sigma_1(\bx(\lambda))
\]
and thus $\lexmax X = \{\bx^*\}$. We also know that $\bx_{c, 0}$ must exist and is unique, as it is defined to be the minimum of a strictly convex function on a line segment. 

We have
\begin{align*}
\sigma_1(\bx(\lambda)) &= \lambda \epsilon,\\
\sigma_{k-1}(\bx(\lambda)) &= \lambda \epsilon + (1 - \lambda)\beta,\\
\sigma_k(\bx(\lambda)) &= \lambda + (1 - \lambda)\beta .
\end{align*}
Therefore
\begin{align*}
L_c(\bx(\lambda)) &= \sum_{i = 1}^n \exp(-c\sigma_i(\bx(\lambda)))\\
&= \exp(-c\lambda\epsilon) + (k-3)\exp(-c\eps) + \exp\left(- c\lambda\epsilon - c (1 - \lambda)\beta \right)\\
&~~~~~ + \exp\left(-c\lambda - c(1 - \lambda)\beta\right) + (n - k) \exp(-c).
\end{align*}
Let $\ell(\lambda) = L_c(\bx(\lambda))$ for $\lambda \in [0, 1]$. We have
\begin{align*}
\ell'(\lambda) &=  -\exp(-c\lambda\epsilon)c\epsilon - \exp\left(-c\lambda\epsilon - c(1 - \lambda)\beta\right)\left(c\epsilon - c\beta\right)\\
&~~~~~ -\exp\left(-c\lambda - c(1 - \lambda)\beta\right)\left(c - c\beta\right)
\end{align*}
and
\begin{align*}
\ell''(\lambda) &=  \exp(-c\lambda\epsilon)(c\epsilon)^2 + \exp\left(-c\lambda\epsilon - c(1 - \lambda)\beta\right)\left(c\epsilon - c\beta\right)^2\notag\\
&~~~~~ +\exp\left(-c\lambda - c(1 - \lambda)\beta\right)\left(c - c\beta\right)^2.
\end{align*}
We now divide into two cases, $c \le a$ and $c > a$. First suppose $c \le a$. Since $\ell''(\lambda) \ge 0$ for all $\lambda \in [0, 1]$, we know that $\ell'(\lambda)$ is minimized at $\lambda = 0$. We have
\begin{align*}
    \ell'(0) \ge 0
    &\Leftrightarrow -c\eps - \exp\left(-c\beta\right)\left(c\epsilon - c\beta\right) - \exp(-c\beta)(c - c\beta) \ge 0\\
    &\Leftrightarrow \exp\left(-c\beta\right)\left(2c\beta - c - c\epsilon\right) \ge c\epsilon\\
    &\Leftrightarrow \exp\left(-c\beta\right) \ge \frac{\epsilon}{2\beta - 1 - \epsilon}\\
    &\Leftrightarrow c \le \frac{1}{\beta} \log \frac{2\beta - 1 - \eps}{\eps}\\
    &\Leftrightarrow c \le a.
\end{align*}
Thus if $c \le a$ then $\ell'(\lambda) \ge 0$ for all $\lambda \in [0, 1]$, which implies that $\ell(\lambda)$ is minimized at $\lambda = 0$. In other words, $\bx_{c,0} = \bx'$, and therefore
\begin{equation}
\sigma_k(\bx_{c,0}) - \sigma_k(\bx^*) = \sigma_k(\bx') - \sigma_k(\bx^*) = \beta - 1. \label{eq:bound1}
\end{equation}

Now suppose $c > a$, which by the above derivation implies that $\ell'(0) < 0$. We also have $\ell'(1) > 0$, since
\begin{align}
\ell'(1) > 0 &\Leftrightarrow -\exp(-c\eps)c\eps - \exp(-c\eps)(c\eps - c\beta) - \exp(-c)(c - c\beta) > 0 \notag\\
&\Leftrightarrow -c\eps - (c\eps - c\beta) - \exp(-c(1 - \eps))(c - c\beta) > 0 \label{eq:deriv-1}\\
&\Leftrightarrow \frac{\beta - 2\eps}{1 - \beta} > \exp(-c(1 - \eps)) \notag\\
&\Leftrightarrow c > \frac{1}{1 - \eps} \log \frac{1 - \beta}{\beta - 2\eps} \notag
\end{align}
and
\begin{equation}
\frac{1}{1 - \eps} \log \frac{1 - \beta}{\beta - 2\eps} \le \frac87 \log \frac{\frac13}{\frac23 - \frac14} < 0 < a < c\label{eq:deriv0}
\end{equation}
where in Eq.~\eqref{eq:deriv-1} we divided through by $\exp(-c\eps)$ and in Eq.~\eqref{eq:deriv0} we used $\beta = \frac23$ and $\eps \le \frac18$. Therefore $\ell'(\lambda) = 0$ for some $\lambda \in (0, 1)$, which must also be the minimizer of $\ell(\lambda)$. We have
\begin{align}
\ell'(\lambda) = 0 & \Leftrightarrow -\exp(-c\lambda\epsilon)c\epsilon - \exp\left(-c\lambda\epsilon - c(1 - \lambda)\beta\right)\left(c\epsilon - c\beta\right) \notag\\
&~~~~~~ -\exp\left(-c\lambda - c(1 - \lambda)\beta\right)\left(c - c\beta\right) = 0 \notag\\
& \Leftrightarrow -c\epsilon - \exp\left(- c(1 - \lambda)\beta\right)\left(c\epsilon - c\beta\right) -\exp\left(-c(1 - \lambda)\beta\right)\left(c - c\beta\right)z = 0 \label{eq:deriv1}\\
&\Leftrightarrow \exp(-c(1 - \lambda)\beta)\big(c\beta - c\eps - (c - c\beta)z\big) = c\eps \notag\\
&\Leftrightarrow \exp(-c(1 - \lambda)\beta) = \frac{\eps}{\beta - \eps - (1 - \beta)z} \notag\\
&\Leftrightarrow \lambda = 1 - \frac{1}{c\beta} \log \frac{\beta - \epsilon - (1 - \beta)z}{\epsilon}. \label{eq:deriv2}
\end{align}
where in Eq.~\eqref{eq:deriv1} we divided through by $\exp(-c\lambda\eps)$ and let $z = \exp(-c\lambda(1 - \eps))$.
Thus $\bx_{c,0} = \bx(\lambda_c)$, where $\lambda_c$ is the value of $\lambda$ that satisfies Eq.~\eqref{eq:deriv2}. We have
\begin{align}
\sigma_k(\bx_{c,0}) - \sigma_k(\bx^*) &= \sigma_k(\bx(\lambda_c)) - \sigma_k(\bx^*) \notag\\ 
&= \lambda_c + \beta (1 - \lambda_c) - 1 \notag\\
&= 1 - \frac{1}{c\beta} \log \frac{\beta - \epsilon - (1 - \beta)z}{\epsilon} + \frac{1}{c} \log \frac{\beta - \epsilon - (1 - \beta)z}{\epsilon} - 1\notag\\
&= -\left(\frac{1}{\beta} - 1\right) \frac1c \log \frac{\beta - \epsilon - (1 - \beta)z}{\epsilon}\notag\\
&\le -\left(\frac{1}{\beta} - 1\right) \frac1c \log \frac{2\beta  - 1 - \epsilon}{\epsilon} \label{eq:deriv3}\\
&= \left(\beta - 1\right) \frac{1}{c\beta} \log \frac{2\beta  - 1 - \epsilon}{\epsilon}\notag\\
&= \left(\beta - 1\right) \frac{a}{c} \label{eq:bound2}
\end{align}
where in Eq.~\eqref{eq:deriv3} we used $z \le 1$. Combining Eq.~\eqref{eq:bound1} and \eqref{eq:bound2} we have
\[
\sigma_k(\bx_{c,0}) \le \sigma_k(\bx^*) + \max\left\{\beta - 1, (\beta - 1)\frac{a}{c}\right\} = \sigma_k(\bx^*) - \frac13\min\left\{1, \frac{a}{c}\right\} 
\]
where we used $\beta = \frac23$.

\section{Comparison to \citet{hartman2023leximin}}
\label{sec:hartman}

In this paper we consider $\bx \in X$ to be $\eps$-close to a lexicographic maximum of $X$ if $\sigma_i(\bx) \ge \sigma_i(\bx^*) - \eps$ for all $i \in [n]$ and $\bx^* \in \lexmax X$ (see Definition \ref{defn:distort}). However, \citeauthor{hartman2023leximin} take a different approach. For any $\bx, \by \in X$ they say that $\by \succ_\eps \bx$ if there exists $k \in [n]$ such that
\begin{align*}
\sigma_i(\by) &\ge \sigma_i(\bx) ~~\textrm{for all }i < k\\
\sigma_k(\by) &> \sigma_k(\bx) + \eps
\end{align*}
and go on to define $\bx \in X$ to be $\eps$-close to a lexicographic maximum of $X$ if there is no $\by \in X$ such that $\by \succ_\eps \bx$ (see their Section 3.2, and in that section let $\alpha = 1$). These definitions are incompatible. Consider 
\[
X = \left\{\bx_1 = \begin{pmatrix} 10\\ 1\\ 1 \end{pmatrix}, \bx_2 = \begin{pmatrix} 10 - \eps\\ 1 - \eps\\ 1 - \eps \end{pmatrix}, \bx_3 = \begin{pmatrix} 5\\ 5\\ 1 - \eps \end{pmatrix}\right\}
\]
Note that $\bx_1$ is the only lexicographic maximum of $X$. So only $\bx_1$ and $\bx_2$ are $\eps$-close to a lexicographic maximum by our definition, and only $\bx_1$ and $\bx_3$ are $\eps$-close to a lexicographic maximum by their definition.

\end{document}